\documentclass{article}
\usepackage{amsmath, amsthm, amssymb, amsopn, amscd, mathtools, array, enumerate, amsfonts, graphics, hyperref, wasysym, chngcntr, linguex, graphicx}
\usepackage{doc,authblk}
\usepackage{amsfonts}
\usepackage{footnote}
\usepackage{multicol}
\usepackage{booktabs}
\usepackage[flushleft]{threeparttable}
\usepackage[font=small,labelfont=bf]{caption}
\usepackage{graphicx}
\usepackage{tikz}
\usepackage{caption}
\usepackage{subcaption}
\usepackage{float}

\newtheorem{theorem}{Theorem}[section]
\newtheorem{corollary}[theorem]{Corollary}

\newtheorem{lemma and definition}[theorem]{Lemma and Definition}
\newtheorem{proposition}[theorem]{Proposition}

\newtheorem{exam}[theorem]{Example}

\newtheorem{the construction}[theorem]{THE CONSTRUCTION}

\providecommand{\keywords}[1]{\textbf{\text{Keywords:}} #1}
\providecommand{\subjclass}[1]{\textbf{\text{Mathematics Subject Classification (2010):}} #1}

\title{BETTI ELEMENTS AND CATENARY DEGREE OF TELESCOPIC NUMERICAL SEMIGROUP FAMILIES}
\author{Meral S\"{U}ER}
\affil{Department of Mathematics, Faculty of Science and Letters, Batman University, Batman, Turkey \href{mailto:meral.suer@batman.edu.tr}{meral.suer@batman.edu.tr}\footnote{Corresponding Author's E-mail}}

\author{Mehmet \c{S}irin SEZG\.{I}N}
\affil{Department of Mathematics, Institue of Science
	and Letters, Batman University, Batman, Turkey
 \href{mailto:mehmetsezgin1984@gmail.com}{mehmetsezgin1984@gmail.com}} 

\date{}

\begin{document}

\maketitle

\begin{abstract}
The catenary degree is an invariant that measures the distance between factorizations of elements within a numerical semigroup.  In general, all possible catenary degrees of the elements of the numerical semigroups occur as the catenary degree of one of its Betti elements. In this study, Betti elements of some telescopic numerical semigroup families with embedding dimension three were found and formulated. Then, with the help of these formulas, Frobenius numbers and genus of these families were obtained. Also, the catenary degrees of telescopic numerical semigroups were found with the help of factorizations of Betti elements of these semigroups
\end{abstract}
\subjclass{Primary 20M14; Secondary 20M30} \\
\keywords{Betti element, telescopic numerical semigroups, catenary degree, factorizations.}

\section{Introduction}\label{sec1}
Let $\mathbb{N}$ be the set of nonnegative integers. A numerical semigroup is a nonempty subset $\mathbb{S}$ of $\mathbb{N}$ that is closed under addition, contains the zero element, and whose complement in $\mathbb{N}$ is finite. If $a_{1},\ldots,a_{e}$ are positive integers with $\gcd\{a_{1},\ldots,a_{e}\}=1$, then the set $\langle a_{1},\ldots,a_{e}\rangle=\{\lambda_{1}a_{1}+ \cdots+\lambda_{e}a_{e}:\lambda_{1}, \cdots,\lambda_{e}\in\mathbb{N}\}$  is a numerical semigroup. Every numerical semigroup is in this form \cite{7Rosalesgarcia}.
Let $\mathbb{S}$ be a numerical semigroup and $A=\{a_{1},a_{2},\ldots,a_{e}\}\subset\mathbb{N}$ such that $\gcd\{a_{1},a_{2},\ldots,a_{e}\} =1$, where $e\geq1$ and $a_{1}<a_{2}< \cdots<a_{e}$. The set $A$ is a system of generators of $\mathbb{S}$ if there are $\eta_{1},\eta_{2},\ldots,\eta_{e}\in\mathbb{N}$ for all $s\in\mathbb{S}$ and $s={\overset{e}{\underset{i=1}{{\displaystyle\sum}}}\eta_{i}a_{i}}$. If there are no $\eta_{1},\ldots,\eta_{k-1}$ elements in the form $a_{k}={\overset{k-1}{\underset{i=1}{{\displaystyle\sum}}}\eta_{i}a_{i}}$ for all $k=2,\ldots,e,$ then the set $A=\{a_{1},a_{2},\ldots,a_{e}\}$ is a minimal system of generators and we  denote  by $\mathbb{S}=\langle A\rangle$. Let $A=\{a_{1},a_{2},\ldots,a_{e}\}$  be the set of minimal generators of the numerical semigroup $\mathbb{S}$. Then the number $a_{1}$ is called multiplicity of $\mathbb{S}$, denoted  by $\mu(\mathbb{S})$, and the cardinality of $A$ is called embedding dimension of $\mathbb{S}$, denoted  by $e(\mathbb{S})$ \cite{7Rosalesgarcia}.\\ 	\indent If $\mathbb{S}$ is a numerical semigroup, the largest integer not belonging to $\mathbb{S}$ is called Frobenius number of $\mathbb{S}$, denoted by $F(\mathbb{S})$. We say that a positive integer $x$ is a gap of $\mathbb{S}$ if $x\notin\mathbb{S}$. The set of all the gaps of $\mathbb{S}$ is denoted by $G(\mathbb{S})$. The cardinality of $G(\mathbb{S})$ is called the genus of $\mathbb{S}$, denoted by $g(\mathbb{S})$ \cite{1Blanco}. \\ 	\indent Let $\mathbb{S}$ be a numerical semigroup  minimally generated by $\{a_{1},\ldots,a_{e}\}$. The homomorphism

$$\phi:\mathbb{N}^{e}\rightarrow\mathbb{S}, \qquad \phi(\eta_{1},\ldots,\eta_{e})=\eta_{1}a_{1}+\ldots+\eta_{1}a_{e}$$
is the factorization homomorphism of $\mathbb{S}$.  Then $\mathbb{S}$ is isomorphic to $\mathbb{N}^{e}/ker\phi$, where $ker\phi$ is kernel congruence, which means that $(x,y)\in ker\phi$ if $\phi(x)=\phi(y)$ \cite{3Conaway}. The set of factorizations of $a\in\mathbb{S}$ is defined by
$$Z(a)=\phi^{-1}(a)=\{(\eta_{1},\ldots,\eta_{e})\in\mathbb{N}^{e}:\eta_{1}a_{1}+\cdots+\eta_{1}a_{e}=a\}.$$
If a factorization has a positive entry in the e-tuple, we say that the element is supported on the component corresponding to that generator. The length of $a$ is $|a|=\eta_{1}+\ldots+\eta_{e}$. For $x,y\in \mathbb{N}^{e}$, with $x=(x_{1},\ldots,x_{e})$ and $y=(y_{1},\ldots,y_{e})$. The greatest common divisor of $x$ and $y$ is defined as
$$\gcd\{x,y\}=(\min\{x_{1},y_{1}\}, \cdots,\min\{x_{e},y_{e}\}). $$
The distance between $x$ and $y$ is
$$dist\{x,y\}=\max\{|x-\gcd\{x,y\}|,|y-\gcd\{x,y\}|\}.$$
Given a positive integer $N$, an $N$-chain of factorizations from $x$ to $y$ is a sequence $k_{0},\ldots,k_{e}\in Z(a)$ such that $x=k_{0},y=k_{e}$ and $dist\{k_{i},k_{i+1}\}\leq N$ for all $i$. The catenary degree of $a$, denoted by $c(a)$, is the minimal $N\in\mathbb{N}\cup\{\infty\}$ such that for any two factorizations $x,y\in Z(a)$, there is an $N$-chain from $x$ to $y$.

Fix a finitely generated numerical semigroup $\mathbb{S}$. For each $a\in\mathbb{S}\setminus \left\lbrace 0\right\rbrace $, consider the graph $\nabla_{a}$
with vertex set $Z(a)$ in which two vertices $x,y\in Z(a)$, share an edge if $\gcd\{x,y\} \neq 0$. If $\nabla_{a}$ is not
connected, then $a$ is called a Betti element of $\mathbb{S}$. We write
$$Betti(\mathbb{S})=\{y\in\mathbb{S}:\nabla_{a} \text{ is disconnected}\}$$
for the set of Betti elements of $\mathbb{S}$.

Let $\mathbb{S}$ be numerical semigroup and $a\in\mathbb{S}$, $x,y\in Z(a)$ and $N\in\mathbb{N}$. In this case, the catenary degree of $a$, denoted by $c(a)$, is the smallest of the existing $N$-chains. Furthermore, the set of catenary degrees of $\mathbb{S}$ is the set $C(\mathbb{S})=\{c(s):s\in\mathbb{S}\}$, and the catenary degree of $\mathbb{S}$ is the supremum of this set, namely $c(\mathbb{S})= \sup{C(\mathbb{S})}$ \cite{1Assi,2Chapman,6CONeil}.\\ 	\indent Calculating the Beti elements of a numerical semigroup have complex properties. 
It is known that the maximum catenary degree of the numerical semigroup is reached with the help of an element called the Betti element. Also, it knows that the numerical semigroups with embedding dimension three have at most three Betti elements \cite{3Conaway}.\\ 	\indent Let $(a_{1},\ldots,a_{e})$ be a sequence of positive integers with $a_{1}< \cdots<a_{e}$ and such that their greatest common divisor is $1$. Define $d_{i}=\gcd\{a_{1},\ldots,a_{i}\}$ and $A_{i}=\{a_{1}/d_{1},\ldots,a_{i}/d_{i}\}$  for $i=1,\ldots,e$. Let $\mathbb{S}_{i}$ be the semigroup generated by $A_{i}$. If $a_{i}/d_{i}\in\mathbb{S}_{i-1}$ for $i=2,\ldots,e$, we call that the sequence $(a_{1},\ldots,a_{e})$ is telescopic. A numerical semigroup is telescopic if it is generated by a telescopic sequence \cite{5Micale}. Specially, let $\langle a_{1},a_{2},a_{3}\rangle$ be a numerical semigroup. If $a_{3}\in\langle a_{1}/d,a_{2}/d\rangle$, then $\mathbb{S}$ is called triply-generated telescopic semigroup, where $d=\gcd\{a_{1},a_{2}\}$ \cite{4Matthews}.\\ 	\indent 
An element of a numerical semigroup is expressed in different ways as a linear combination with non-negative integer coefficients of its generators. This expression is known as a factorization of that element. The catenary degree of the element of the numerical semigroup is a combinatorial constant that describes the relationships between differing irreducible factorizations of the
element. The supremum of all catenary degrees of all the elements in the numerical semigroup is the catenary degree of the numerical semigroup itself. While the set of factorizations for a numerical semigroup is a perfect invariant, it is often stodgy to compute and encode. We focus on invariants obtained by passing from factorizations to their lengths. Many of the arguments that compute the maximal catenary degree $c(\mathbb{S})$ for a numerical semigroup $\mathbb{S}$ focus on the Betti elements of $\mathbb{S}$. The factorizations of Betti elements contain enough information about the set of factorizations of $\mathbb{S}$ to give sharp bounds on the catenary degrees occurring in $C(\mathbb{S})$.\\ 	\indent 

Conaway et al., O’Neil et al. and Chapman et al. presented a study on how to find some Betti elements in a numerical semigroup with three generations in their works\cite{2Chapman,3Conaway,6CONeil}. Süer and Ilhan have found and proved some telescopic numerical semigroup families in their work\cite{8Suer}.\\ 	\indent 

The paper is organized as follows. In Section \ref{sec2}, we will find the Betti elements of the telescopic numerical semigroup families found by Süer and Ilhan \cite{8Suer}, using the advantages of \cite{2Chapman}, \cite{3Conaway}, and \cite{6CONeil}. Later we will find and prove some formulas for the Betti elements of these telescopic numerical semigroup families. Furthermore, with the help of the obtained results, we will find some formulas for the Frobenius numbers and genus of these families. In Section \ref{sec3}, we will obtain the factorizations of Betti elements of these semigroups. And we will calculate the catenary degrees of telescopic numerical semigroups with the help of factorizations of Betti elements of these semigroups.
\section{The set of Betti elements of the telescopic numerical semigroup families}\label{sec2}
In this section, we will find the set of Betti elements of the telescopic numerical semigroup families obtained in \cite{8Suer}.
\begin{proposition}\label{prop1} [\cite{3Conaway}, Proposition 4.1]
	Let $\mathbb{S}=\langle u_{1},u_{2},u_{3}\rangle$ be a numerical semigroup minimally generated. An element $\beta\in\mathbb{S}$ is a Betti element if $\beta=x_{i}u_{i} $ for some $i=1,2,3$ where $x_{i}=\min\{x:xu_{i}\in\langle a_{j},a_{k}\rangle\ \text{ where} \{j,k\}=\{1,2,3\}\backslash \{i\}\}$.	
\end{proposition}
\begin{theorem}\label{thm1}[\cite{8Suer}, 2.4. Theorem]
	Let $\mathbb{S}$ be a numerical semigroup with embedding dimension three and multiplicity four. The numerical semigroup $\mathbb{S}$ is telescopic if only if $\mathbb{S}$ is  a member of the family $\Phi=\{\langle4,4\alpha+2,m\rangle : \alpha\in \mathbb{N}, \quad m\in \mathbb{N}_{o} \text{ and } m>4\alpha+2 \}$ ( where $\mathbb{N}_{o}$ denotes the set of positive odd integers ).
\end{theorem}
\begin{theorem}\label{thm2}[\cite{8Suer}, 2.7. Theorem]
	Let $\mathbb{S}$ be a numerical semigroup with embedding dimension three and multiplicity six. The numerical semigroup $\mathbb{S}$ is telescopic if only if $\mathbb{S}$ is a member of the following families:
	\flushleft
	\begin{itemize}
		\item[i)] $\prod=\{\langle6,6\alpha+2,m\rangle : \alpha\in \mathbb{N}, \quad m\in \mathbb{N}_{o} \text{ and }m>6\alpha+2\},$
		\item[ii)] $\Omega=\{\langle6,6\alpha+3,n\rangle : \alpha,n\in \mathbb{N}, \quad 3\nmid n \text
		{ and }n>6\alpha+3\},$
		\item[iii)] $\Psi=\{\langle6,6\alpha+4,p\rangle : \alpha\in \mathbb{N}, \quad p\in \mathbb{N}_{o} \text{ and }p>6\alpha+4\}$.
	\end{itemize}
\end{theorem}
\begin{theorem}\label{thm3}
	If $\mathbb{S}$ be a member of the telescopic numerical semigroup family $\Phi$ given in Theorem \ref{thm1}, then the set of Betti elements of $\mathbb{S}$ is $Betti(\mathbb{S})=\{8\alpha+4,2m\}$.
\end{theorem}
\begin{proof}
	Let $\mathbb{S}$ be an element of the telescopic numerical semigroup family $\Phi$  in Theorem \ref{thm1}. Then, $x_{1}=\min\{x:4x\in\langle4\alpha+2,m\rangle\}$  is written by Proposition \ref{prop1}. Also, $4x=\eta_{1}(4\alpha+2)+\eta_{2}m \quad (\eta_{1},\eta_{2}\in\mathbb{N})$ is written by the definition of the numerical semigroup. To obtain the smallest value of $x$, we must write $\eta_{1}=2$ and $\eta_{2}=0$  in the given equation. Hence, if $4x=2(4\alpha+2)+0m$, then it is obtained as  $x=2\alpha+1$. Then, it is found as  $\beta_{1}=4(2\alpha+1)=8\alpha+4$ by Proposition \ref{prop1}.
	
	Similary, from Proposition \ref{prop1}, $x_{2}=\min\{x:x(4\alpha+2)\in\langle4,m\rangle\}$ is written. Besides, $x(4\alpha+2)=\lambda_{1}4+\lambda_{2}m \quad (\lambda_{1},\lambda_{2}\in\mathbb{N})$   is written by the definition of the numerical semigroup. To obtain the smallest value of $x$, we must write $\lambda_{1}=2\alpha+1$ and $\lambda_{2}=0$  in the given equation. Then, if $4\alpha+2=4(2\alpha+1)+0m$, then it is obtained as $x=2$. Thus, it is found as  $\beta_{2}=2(4\alpha+2)=8\alpha+4=\beta_{1}$ by Proposition \ref{prop1}.
	
	Again, $x_{3}=\min\{x:mx\in\langle4,4\alpha+2\rangle\}$ is written by Proposition \ref{prop1}. Since $\mathbb{S}$ is telescopic numerical semigroup, we can write $\gcd\{4,4\alpha+2\}=2$ and $m\in\langle2,2\alpha+1\rangle$. There are $\mu_{1},\mu_{2}\in\mathbb{N}$ such that $m=2\mu_{1}+(2\alpha+1)\mu_{2}$ by the definition of the numerical semigroup. Hence, it is obtained as $2m=4\mu_{1}+(4\alpha+2)\mu_{2}$. The smallest positive integer $x$ that satisfies these conditions is $2$. Then, it is found as $\beta_{3}=2m$ by Proposition \ref{prop1}.
	
	As a result, the telescopic numerical semigroup $\mathbb{S}$ in the family $\Phi$  has got two different Betti elements. The set of Betti elements of the telescopic numerical semigroup $\mathbb{S}$ in the family $\Phi$  in Theorem \ref{thm1} is in the form $Betti(\mathbb{S})=\{8\alpha+4,2m\}$.
\end{proof}

\begin{corollary}
	Let the numerical semigroup $\mathbb{S}$ be a member of the telescopic numerical semigroup family $\Phi$ in Theorem \ref{thm1}. While $\beta_{1}=\beta_{2}=8\alpha+4$  and $\beta_{3}=2m$,
	\begin{itemize}
		\item[i)] $F(\mathbb{S})=\frac{\beta_{1}+\beta_{3}}{2}-4,$
		\item[ii)] $g(\mathbb{S})=\frac{\beta_{1}+\beta_{3}}{4}-\frac{3}{2}.$
	\end{itemize}
\end{corollary}
\begin{theorem}\label{thm4}
	Let $\mathbb{S}$ be a member of the telescopic numerical semigroup families given in Theorem \ref{thm2}.
	\begin{itemize}
		\item[i)] If $\mathbb{S}$ be a member of the telescopic numerical semigroup family $\prod$ given in Theorem \ref{thm2}, then the set of Betti elements of $\mathbb{S}$ is in the form $Betti(\mathbb{S})=\{18\alpha+6,2m\}$.
		\item[ii)] If $\mathbb{S}$ be a member of the telescopic numerical semigroup family $\Omega$ given in Theorem \ref{thm2}, then the set of Betti elements of $\mathbb{S}$ is in the form $Betti(\mathbb{S})=\{12\alpha+6,3n\}$.
		\item[iii)] If $\mathbb{S}$ be a member of the telescopic numerical semigroup family $\Psi$ given in Theorem \ref{thm2}, then the set of Betti elements of $\mathbb{S}$ is in the form $Betti(\mathbb{S})=\{18\alpha+12,2p\}$.
	\end{itemize}
\end{theorem}
\begin{proof}
	Assume that $\mathbb{S}$ is a member of the telescopic numerical semigroup families given in Theorem \ref{thm2}.
	\begin{itemize}
		\item[i)] If $\mathbb{S}$ be a member of the telescopic numerical semigroup family $\prod$ given in Theorem \ref{thm2}, then $x_{1}=\min\{x:6x\in\langle6\alpha+2,m\rangle\}$ is written by Proposition \ref{prop1}. From definition of the numerical semigroup, we can write $6x=\varphi_{1}(6\alpha+2)+\varphi_{2}m$ where $\varphi_{1},\varphi_{2}\in\mathbb{N}$. To obtain the smallest value of $x$, we obtain two different situations according to the value of $m$, $6\alpha+2=2(3\alpha+1)<m\leqslant3(3\alpha+1)=9\alpha+3$ or $m>9\alpha+3$:
		\begin{itemize}
			\item[a)] If $6\alpha+2<m\leq9\alpha+3$, then again there are two situations: $3|m$ and $3\nmid m$:\\
			If $6\alpha+2<m\leq9\alpha+3$  and $3|m$, then when we chose $\varphi_{1}=0$ and $\varphi_{2}=2$, we obtain the smallest value of. If $6x=0(6\alpha+2)+2m$, then $x=\frac{m}{3}$ is obtained. Then, it is found as $\beta_{1}=6\cdot\frac{m}{3}=2m$. \par 
			If $6\alpha+2<m\leq9\alpha+3$  and $3\nmid m$, then when we chose $\varphi_{1}=3$ and $\varphi_{2}=0$, we obtain the smallest value of $x$. If $6x=3(6\alpha+2)+0m$, then it is obtained as $x=3\alpha+1$. Then, it is found as $\beta_{1}=6(3\alpha+1)=18\alpha+6$.\par 
			
			\item[b)] If $m>9\alpha+3$, then we get the smallest value of $x$ when we choose $\varphi_{1}=3$ and $\varphi_{2}=0$ in the given equation. If $6x=3(6\alpha+2)+0m$, then it is obtained as $x=3\alpha+1$. Then, it is found as  $\beta_{1}=6(3\alpha+1)=18\alpha+6$.\par 	
		\end{itemize}
		From Proposition \ref{prop1}, $x_{2}=\min\{x:(6\alpha+2)x\in\langle6,m\rangle\}$ is written. From definition of the numerical semigroup, we can write $(6\alpha+2)x=\gamma_{1}6+\gamma_{2}m$ where $\gamma_{1},\gamma_{2}\in\mathbb{N}$. when we chose $\gamma_{1}=3\alpha+1$ and $\gamma_{2}=0$, we obtain the smallest value of $x$. Hence, If $(6\alpha+2)x=6(3\alpha+1)+0m$, then $x=3$ is obtained. Then, it is found as  $\beta_{2}=3(6\alpha+2)=18\alpha+6=\beta_{1}$ by Proposition \ref{prop1}..
		
		From Proposition \ref{prop1}, $x_{3}=\min\{x:mx\in\langle6,6\alpha+2\rangle\}$ is written. Since $\mathbb{S}$ is a telescopic numerical semigroup, we can write $\gcd\{6,6\alpha+2\}=2$ and $m\in\langle3,3\alpha+1\rangle$. There are  $\vartheta_{1}, \vartheta_{2}\in\mathbb{N}$ such that $m=3\vartheta_{1}+(3\alpha+1)\vartheta_{2}$  by the definition of the numerical semigroup. Hence, $2m=6\vartheta_{1}+(6\alpha+2)\vartheta_{2}$  is obtained. The smallest nonnegative integer $x$  that satisfies these conditions is $2$. Then, it is found as $\beta_{3}=2m$ by Proposition \ref{prop1}.
		
		As a result, the telescopic numerical semigroup $\mathbb{S}$ in the family $\prod$ given in Theorem \ref{thm2}  has got two different Betti elements at most. The set of Betti elements of the telescopic numerical semigroup $\mathbb{S}$ in the family $\prod$  in Theorem \ref{thm2} is in the form $Betti(\mathbb{S})=\{\beta_{1}=\beta_{2}=18\alpha+6,\beta_{3}=2m\}$ or $Betti(\mathbb{S})=\{\beta_{2}=18\alpha+6,\beta_{1}=\beta_{3}=2m\}$.
		\item[ii)] If $\mathbb{S}$ be a member of the telescopic numerical semigroup family $\Omega$ given in Theorem \ref{thm2}, then $x_{1}=\min\{x:6x\in\langle6\alpha+3,n\rangle\}$  is written by Proposition \ref{prop1}. $6x=\omega_{1}(6\alpha+3)+\omega_{2}n \quad (\omega_{1},\omega_{2}\in\mathbb{N})$    is written by the definition of the numerical semigroup. When we choose $\omega_{1}=2$ and $\omega_{2}=0$ in the given equation, the smallest value of $x$ is obtained. Then, $6x=2(6\alpha+3)+0n$ and $x=2\alpha+1$. Then, it is found as $\beta_{1}=6(2\alpha+1)=12\alpha+6$ by Proposition \ref{prop1}.
		
		Again, from Proposition \ref{prop1}, $x_{2}=\min\{x:(6\alpha+3)x\in\langle6,n\rangle\}$ is written. $(6\alpha+3)x=\psi_{1}6+\psi_{2}n \quad (\psi_{1},\psi_{2}\in\mathbb{N})$ is written by the definition of the numerical semigroup. Thus, when $\psi_{1}=2\alpha+1$ and $\psi_{2}=0$, the smallest value of $x$ is obtained. Hence,  $(6\alpha+3)x=(2\alpha+1)6+0n$ and $x=2$. As a result, $\beta_{2}=2(6\alpha+3)=12\alpha+6=\beta_{1}$ by Proposition \ref{prop1}.
		
		$x_{3}=\min\{x:(nx\in\langle6,6\alpha+3\rangle\}$ is written by Proposition \ref{prop1}. Since $\mathbb{S}$ is a telescopic numerical semigroup, we can write $\gcd\{6,6\alpha+3\}=3$ and $n\in\langle2,2\alpha+1\rangle$. There are $\phi_{1},\phi_{2}\in\mathbb{N}$ such that  $n=2\phi_{1}+(2\alpha+1)\phi_{2}$ by the definition of the numerical semigroup. Thus, $3n=6\phi_{1}+(6\alpha+3)\phi_{2}$ can be written. Then, it is found as $\beta_{3}=3n$ by Proposition \ref{prop1}.
		
		Consequently, the telescopic numerical semigroup $\mathbb{S}$ in the family $\Omega$ given in Theorem \ref{thm2}  has got two different Betti elements at most. The set of Betti elements of the telescopic numerical semigroup $\mathbb{S}$ in the family $\Omega$ in Theorem \ref{thm2} is in the form $Betti(\mathbb{S})=\{\beta_{1}=\beta_{2}=12\alpha+6,\beta_{3}=3n\}$.
		\item[iii)]If $\mathbb{S}$ be a member of the telescopic numerical semigroup family $\Psi$  given in Theorem \ref{thm2}, then $x_{1}=\min\{x:6x\in\langle6\alpha+4,p\rangle\}$ is written by Proposition \ref{prop1}. There are $\eta_{1},\eta_{2}\in\mathbb{N}$ such that $6x=\eta_{1}(6\alpha+4)+\eta_{2}p$ is written by the definition of the numerical semigroup.
		To obtain the smallest value of $x$, we obtain two different situations according to the value of $p$, $6\alpha+4=2(3\alpha+2)<p\leqslant3(3\alpha+2)=9\alpha+6$ or $m>9\alpha+6$:
		\begin{itemize}
			\item[a)] If $6\alpha+4<p\leq9\alpha+6$, then again there are two situations: $3|p$ and $3\nmid p$:\\
			If $6\alpha+4<p\leq9\alpha+6$  and $3|p$, then when we chose $\eta_{1}=0$ and $\eta_{2}=2$, we obtain the smallest value of $x$. If $6x=0(6\alpha+4)+2p$, then $x=\frac{p}{3}$ is obtained. Then, it is found as $\beta_{1}=6\cdot\frac{p}{3}=2p$. \par 
			If $6\alpha+4<p\leq9\alpha+6$  and $3\nmid p$, then when we chose $\eta_{1}=3$ and $\eta_{2}=0$, we obtain the smallest value of $x$. If $6x=3(6\alpha+4)+0p$, then it is obtained as $x=3\alpha+2$. Then, it is found as $\beta_{1}=6(3\alpha+2)=18\alpha+12$.\par 
			
			\item[b)] If $m>9\alpha+6$, then we get the smallest value of $x$ when we choose $\eta_{1}=3$ and $\eta_{2}=0$ in the given equation. If $6x=3(6\alpha+4)+0p$, then it is obtained as $x=3\alpha+2$. Then, it is found as  $\beta_{2}=6(3\alpha+2)=18\alpha+12$.\par 	
		\end{itemize}
		
		$x_{2}=\min\{x:(6\alpha+4)x\in\langle6,p\rangle\}$ is written by Proposition \ref{prop1}. There are $\gamma_{1},\gamma_{2}\in\mathbb{N}$ such that $(6\alpha+4)x=\gamma_{1}6+\gamma_{2}p$ is written by the definition of the numerical semigroup. When we choose $\gamma_{1}=3\alpha+2$ and $\gamma_{2}=0$ in the given equation, the smallest value of $x$ is calculated. Thus, $(6\alpha+4)x=(3\alpha+2)6+0p$ and $x=3$ is obtained. So it is found as $\beta_{2}=3(6\alpha+4)=18\alpha+12=\beta_{1}$ by Proposition \ref{prop1}.
		
		Similarly, $x_{3}=\min\{x:px\in\langle6,6\alpha+4\rangle\}$ is written by Proposition \ref{prop1}. Since $\mathbb{S}$ is a telescopic numerical semigroup, we can write $\gcd\{6,6\alpha+4\}=2$ and $p\in\langle3,3\alpha+2\rangle$. There are $\delta_{1},\delta_{2}\in\mathbb{N}$ such that  $p=3\delta_{1}+(3\alpha+2)\delta_{2}$ by the definition of the numerical semigroup. Thus, $2p=6\delta_{1}+(6\alpha+4)\delta_{2}$ can be written. The smallest positive integer $x$  that satisfies these conditions is $2$. Then, it is found as $\beta_{3}=2p$ by Proposition \ref{prop1}.
		
		As a result, the telescopic numerical semigroup $\mathbb{S}$ in the family $\Psi$ given in Theorem \ref{thm2}  has got two different Betti elements at most. The set of Betti elements of the telescopic numerical semigroup $\mathbb{S}$ in the family $\Psi$ in Theorem \ref{thm2} is in the form $Betti(\mathbb{S})=\{\beta_{1}=\beta_{2}=18\alpha+12,\beta_{3}=2p\}$ or $Betti(\mathbb{S})=\{\beta_{2}=18\alpha+12,\beta_{1}=\beta_{3}=2p\}$.
		
	\end{itemize}
\end{proof}

\begin{corollary}
	Let the numerical semigroup $\mathbb{S}$ be a member of the telescopic numerical semigroup family $\prod$ in Theorem \ref{thm2}. While $\beta_{2}=18\alpha+6$  and $\beta_{3}=2m$,
	\begin{itemize}
		\item[i)] $F(\mathbb{S})=\frac{\beta_{2}+\beta_{3}}{2}-(5-3\alpha),$
		\item[ii)] $g(\mathbb{S})=\frac{\beta_{2}+\beta_{3}}{4}-(\frac{3\alpha-4}{2}).$
	\end{itemize}	
\end{corollary}

\begin{corollary}
	Let the numerical semigroup $\mathbb{S}$ be a member of the telescopic numerical semigroup family $\Omega$ in Theorem \ref{thm2}. While $\beta_{2}=12\alpha+6$  and $\beta_{3}=3n$,
	\begin{itemize}
		\item[i)] $F(\mathbb{S})=\frac{\beta_{2}+\beta_{3}}{2}-(\frac{12-n}{2}),$
		\item[ii)] $g(\mathbb{S})=\frac{\beta_{2}+\beta_{3}}{4}+(\frac{n-10}{4}).$
	\end{itemize}	
\end{corollary}

\begin{corollary}
	Let the numerical semigroup $\mathbb{S}$ be a member of the telescopic numerical semigroup family $\Psi$ in Theorem \ref{thm2}. While $\beta_{2}=18\alpha+12$  and $\beta_{3}=2p$,
	\begin{itemize}
		\item[i)] $F(\mathbb{S})=\frac{\beta_{2}+\beta_{3}}{2}-(4-3\alpha),$
		\item[ii)] $g(\mathbb{S})=\frac{\beta_{2}+\beta_{3}}{4}-(\frac{24p+3p+21}{2}).$
	\end{itemize}	
\end{corollary}
\section{The catenary degree of telescopic numerical semigroups}\label{sec3}
In this section, we will find the factorizations of Betti elements of the telescopic numerical semigroup families obtained in \cite{8Suer}. We will calculate the catenary degrees of these telescopic numerical semigroups.
\begin{theorem}\label{thm5}
	Let $\mathbb{S}$ be a member of the telescopic numerical semigroup family $\Phi$ given in Theorem \ref{thm1}, and $\beta_{1}, \beta_{2},\beta_{3}$ be Betti elements of the numerical semigroup $\mathbb{S}$. In this case, the factorizations of $\beta_{1}=\beta_{2}$ are $(2\alpha+1,0,0)$ and $(0,2,0)$; the factorizations of $\beta_{3}$ are $(\frac{m-k\cdot(2\alpha+1)}{2},k,0)$ and $(0,0,2)$ for $k\in \mathbb{N}_{o}$ and $k\leq\frac{m}{2\alpha+1}$.
\end{theorem}
\begin{proof}
	Assume that $\mathbb{S}$ is a member of the telescopic numerical semigroup family given in Theorem \ref{thm1}. According to the proof of Theorem \ref{thm3}, the Betti elements of $\mathbb{S}$ are respectively $\beta_{1}=\beta_{2}=8\alpha+4$ and $\beta_{3}=2m$. Firstly, we will find the factorizations of $\beta_{1}=\beta_{2}=8\alpha+4$. We write $\beta_{1}=\beta_{2}=4x_{1}+(4\alpha+2)x_{2}+mx_{3} \quad (x_{1},x_{2},x_{3}\in\mathbb{N})$ by definition the factorizations. In this case, since $\beta_{1}=\beta_{2}=8\alpha+4$ is a positive even  integer,  $\beta_{1}=\beta_{2}=4x_{1}+(4\alpha+2)x_{2}+mx_{3}$  must be a positive even integer, too. Since $m$ is a positive odd integer, $x_{3}$  must be a nonnegative even  integer. Furthermore, it should be $x_{3}=0$ since $m>4\alpha+2$. Thus, $8\alpha+4=4x_{1}+(4\alpha+2)x_{2}$ is obtained.  In this case, since  $x_{2}=2-\frac{2x_{1}}{2\alpha+1}$  and $x_{2}\in\mathbb{N}$, $x_{1}=0$  or $x_{1}=2\alpha+1$ is obtained. As a result, the factorizations of $\beta_{1}=\beta_{2}$ are $(2\alpha+1,0,0)$ and $(0,2,0)$.
	
	Now, we will find the factorizations of $\beta_{3}=2m$. We write $\beta_{3}=2m=4x_{1}+(4\alpha+2)x_{2}+mx_{3} \quad (x_{1},x_{2},x_{3}\in\mathbb{N})$ by definition the factorizations. From here it is clear that $x_{3}=0$ or $x_{3}=1$ or $x_{3}=2$. If $x_{3}=0$, then  $2m=4x_{1}+(4\alpha+2)x_{2}$ and $x_{2}=\frac{m-2x_{1}}{2\alpha+1}$ are obtained. In this case, since $x_{2}\in\mathbb{N}$, we write $2\alpha+1|m-2x_{1}$. Also, since $2\alpha+1$  and $m-2x_{1}$  are odd integers, $x_{2}$ must be an odd integer. Thus, $x_{1}=\frac{m-k\cdot(2\alpha+1)}{2}$  and $k\leq\frac{m}{2\alpha+1}$ for $x_{2}=k\in \mathbb{N}_{o}$. So, if $x_{3}=0$, the factorization of $\beta_{3}$ is $(\frac{m-k\cdot(2\alpha+1)}{2},k,0)$ for $k\in \mathbb{N}_{o}$ and $k\leq\frac{m}{2\alpha+1}$. If $x_{3}=1$, then  $m=4x_{1}+(4\alpha+2)x_{2}$ is obtained. But this contradicts the fact that $m$ is a positive odd integer. If $x_{3}=2$, then $4x_{1}+(4\alpha+2)x_{2}=0$ is obtained. Since $x_{1}$ and $x_{2}$ are nonnegative integers, $x_{1}$ and $x_{2}$  must be $0$. Thus, the factorizations of $\beta_{3}$ is $(0,0,2)$.
\end{proof}
\begin{theorem}\label{thm6} Let $\mathbb{S}$ be a member of the telescopic numerical semigroup family $\Phi$ given in Theorem \ref{thm1}. The catenary degree of Betti elements of $\mathbb{S}$ is following that:
	\begin{itemize}
		\item[i)] $c(\beta_{1})=c(\beta_{2})=c(8\alpha+4)=2\alpha+1$
		\item[ii)] $c(\beta_{3})=c(2m)=\max\{2\alpha+1,\frac{m-\max\{k\}\cdot(2\alpha-1)}{2}\}$ for $k\leq\frac{m}{2\alpha+1}$ and $k\in \mathbb{N}_{o}$.
	\end{itemize}
\end{theorem}
\begin{proof} Assume that $\mathbb{S}$ is a member of the telescopic numerical semigroup family $\Phi$ given in Theorem \ref{thm1}. The factorizations of the Betti elements of the numerical semigroup $\mathbb{S}$ are given in Theorem \ref{thm5}.	
	
	\begin{itemize}
		\item[i)] The factorizations of $\beta_{1}=\beta_{2}=8\alpha+4$ are $(2\alpha+1,0,0)$ and $(0,2,0)$. In this case, the distance of the edge between these factorizations is found as 
		$$\gcd\{(2\alpha+1,0,0),(0,2,0)\}=(0,0,0)$$ 
		and
		$$dist\{(2\alpha+1,0,0),(0,2,0)\}=2\alpha+1.$$ 
		The factorization points that we found are vertices and the path that are connecting these vertices are edges. Hence, if we draw the graph in Figure \ref{fig:1} which consists of  these vertices and edge, the catenary degree of $\beta_{1}=\beta_{2}=8\alpha+4$ is $2\alpha+1$.
		\\
		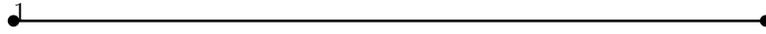
\begin{figure}[H]
			\centering
			\begin{tikzpicture}
			\draw[thick,-] (0,0)--(10,0);
			\draw[thick,-] (10,0)--(0,0);
			\filldraw[black] (0,0) circle (2pt);
			\filldraw[black] (10,0) circle (2pt);
			\put(-20,5){$\left( 2\alpha+1, 0, 0 \right)$};
			\put(270,5){$\left( 0, 2, 0 \right)$};
			\put(150,5){$2\alpha +1$};
			\end{tikzpicture}
			\vspace{0.5 cm}
			\caption{The catenary graph of $\beta_{1}=\beta_{2}=8\alpha+4$} \label{fig:1}
		\end{figure}
		
		\item[ii)] we will find the catenary degree of $\beta_{3}=2m$. From Theorem \ref{thm5},	the factorizations of $\beta_{3}=2m$ are  $(\frac{m-k\cdot(2\alpha+1)}{2},k,0)$ and $(0,0,2)$ for $k\leq\frac{m}{2\alpha+1}$ and $k\in \mathbb{N}_{o}$. In this case, the distances of the edges between these factorizations are found as follows: \\ 	\indent 	Since there will be an edge for every nonnegative integer $k$, let's show that the edge corresponding to each $k_{i}$ with $a_{i}$ such that $a_{i}=(\frac{m-k_{i}\cdot(2\alpha+1)}{2},k_{i},0)$ for $i\in \left\lbrace 1,2,\dots,n\right\rbrace $. Where $k_{1}<k_{2}<\dots<k_{n}$
		$$\gcd\{a_{i},(0,0,2)\}=(0,0,0)$$
		and
		$$dist\{a_{i},(0,0,2)\}=\frac{m-k_{i}\cdot(2\alpha+1)}{2}+k_{i}=\frac{m-k_{i}\cdot(2\alpha-1)}{2}$$
		Let $i\in \left\lbrace 1,2,\dots,n-1\right\rbrace $ and $j\in \left\lbrace 2,3,\dots,n\right\rbrace $ such that $i<j $.	
		$$\gcd\{a_{i},a_{j}\}=(\frac{m-k_{j}\cdot(2\alpha+1)}{2},k_{i},0)$$
		and
		$$dist\{a_{i},a_{j}\}=\max\{|\frac{(k_{j}-k_{i})\cdot(2\alpha+1)}{2}|,|k_{j}-k_{i}|\}=\frac{(k_{j}-k_{i})\cdot(2\alpha+1)}{2}$$
		Moreover, it is clear that  $k_{i}=k_{1}+2(i-1)$ and $|\frac{(k_{j}-k_{i})\cdot(2\alpha+1)}{2}|>|k_{j}-k_{i}|$.
		The following equations can  easily be seen from these obtained above:
		\begin{multline*}
		\min\{dist\{a_{1},(0,0,2)\},dist\{a_{2},(0,0,2)\},\dots,dist\{a_{n},(0,0,2)\}\} = \min\{\frac{m-k_{1}\cdot(2\alpha-1)}{2}, \\
		\frac{m-k_{2}\cdot(2\alpha-1)}{2},\dots,\frac{m-k_{n}\cdot(2\alpha-1)}{2}\}=\frac{m-k_{n}\cdot(2\alpha-1)}{2}=dist\{a_{n},(0,0,2)\}
		\end{multline*}
		for $i\in \left\lbrace 1,2,\dots,n\right\rbrace $
		\begin{multline*}
		\min\{dist\{a_{i},a_{1}\},dist\{a_{i},a_{2}\},\dots,dist\{a_{i},a_{i-1}\},dist\{a_{i},a_{i+1}\},\dots,dist\{a_{i},a_{n}\}\}= \\ \min\{|\frac{(k_{i}-k_{1})\cdot(2\alpha+1)}{2}|,|\frac{(k_{i}-k_{2})\cdot(2\alpha+1)}{2}|,\\ \dots,|\frac{(k_{i}-k_{i-1})\cdot(2\alpha+1)}{2}|,|\frac{(k_{i+1}-k_{i})\cdot(2\alpha+1)}{2}|,\\
		\dots,|\frac{(k_{n}-k_{i})\cdot(2\alpha+1)}{2}|\} =|\frac{(k_{i}-k_{i-1})\cdot(2\alpha+1)}{2}|\\=|\frac{(k_{i+1}-k_{i})\cdot(2\alpha+1)}{2}|
		=\frac{[(k_{1}+(i-1)\cdot2)-(k_{1}+(i-1-1)\cdot2)]\cdot(2\alpha+1)}{2}\\=\frac{[(k_{1}+(i+1-1)\cdot2)-(k_{1}+(i-1)\cdot2)]\cdot(2\alpha+1)}{2}=2\alpha+1
		\end{multline*}
		
		When each vertex is labeled with one of the factorizations of $\beta_{3}=2m$ and each edge is labeled with distance between the factorizations of  $\beta_{3}=2m$ at either end, we get  Figure \ref{fig:2 } $(a)$. If vertices with maximal length are removed from the connected graph in Figure \ref{fig:2 } $(a)$, then Figure \ref{fig:2 } $(b)$ is obtained. So, the catenary degree of $\beta_{3}=2m$ is $\max\left( 2\alpha+1,\frac{m-k_{n}\cdot(2\alpha-1)}{2}\right) $, where $k_{n}=\max\{k\}$ for $k\leq\frac{m}{2\alpha+1}$ and $k\in \mathbb{N}_{o}$.
		\\
		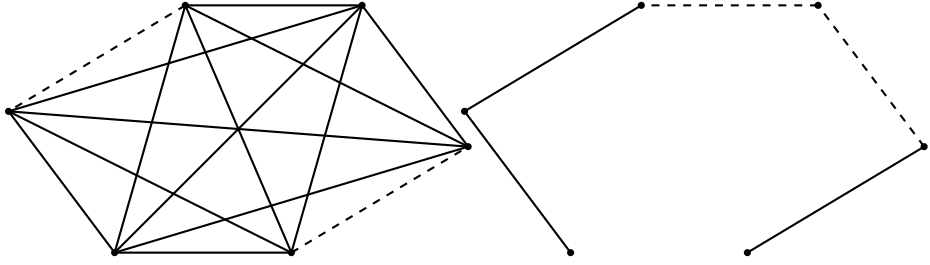
\begin{figure}[!htb]
			\centering
			\begin{minipage}{.5\textwidth}
				\centering
				\centering
				\begin{tikzpicture}[thick,scale=0.47, every node/.style={scale=0.6}]
				\draw[thick,-] (0,4)--(3,0);
				\draw[thick,-] (3,0)--(8,0);
				\begin{scope}[very thick,dashed]
				\draw[thick,-] (8,0)--(13,3);
				\end{scope}
				\draw[thick,-] (13,3)--(10,7);
				\draw[thick,-] (10,7)--(5,7);
				\begin{scope}[very thick,dashed]
				\draw[thick,-] (5,7)--(0,4);
				\end{scope}
				\draw[thick,-] (0,4)--(8,0);
				\draw[thick,-] (0,4)--(13,3);
				\draw[thick,-] (0,4)--(10,7);
				\draw[thick,-] (3,0)--(13,3);
				\draw[thick,-] (3,0)--(10,7);
				\draw[thick,-] (3,0)--(5,7);
				\draw[thick,-] (8,0)--(10,7);
				\draw[thick,-] (8,0)--(5,7);
				\draw[thick,-] (13,3)--(5,7);
				\filldraw[black] (0,4) circle (2pt);
				\filldraw[black] (3,0) circle (2pt);
				\filldraw[black] (8,0) circle (2pt);
				\filldraw[black] (13,3) circle (2pt);
				\filldraw[black] (10,7) circle (2pt);
				\filldraw[black] (5,7) circle (2pt);
				\put(10,-10){\resizebox{.12\hsize}{!}{$(0,0,2)$}};
				\put(100,-10){\resizebox{.06\hsize}{!}{$a_{1}$}};
				\put(170,30){\resizebox{.09\hsize}{!}{$a_{i-1}$}};
				\put(150,90){\resizebox{.06\hsize}{!}{$a_{i}$}};
				\put(40,95){\resizebox{.09\hsize}{!}{$a_{i+1}$}};
				\put(-20,50){\resizebox{.06\hsize}{!}{$a_{n}$}};
				\put(50,-20){\resizebox{.06\hsize}{!}{$(a)$}};	
				\end{tikzpicture}
			\end{minipage}%
			\begin{minipage}{0.5\textwidth}
				
				\begin{tikzpicture}[thick,scale=0.47, every node/.style={scale=0.6}]
				\draw[thick,-] (0,4)--(3,0);
				\draw[thick,-] (8,0)--(13,3);
				\begin{scope}[very thick,dashed]
				\draw[thick,-] (13,3)--(10,7);
				\end{scope}
				\begin{scope}[very thick,dashed]
				\draw[thick,-] (10,7)--(5,7);
				\end{scope}
				\draw[thick,-] (5,7)--(0,4);
				\filldraw[black] (0,4) circle (2pt);
				\filldraw[black] (3,0) circle (2pt);
				\filldraw[black] (8,0) circle (2pt);
				\filldraw[black] (13,3) circle (2pt);
				\filldraw[black] (10,7) circle (2pt);
				\filldraw[black] (5,7) circle (2pt);
				\put(10,-10){\resizebox{.12\hsize}{!}{$(0,0,2)$}};
				\put(100,-10){\resizebox{.06\hsize}{!}{$a_{1}$}};
				\put(170,30){\resizebox{.06\hsize}{!}{$a_{2}$}};
				\put(150,90){\resizebox{.06\hsize}{!}{$a_{i}$}};
				\put(40,95){\resizebox{.09\hsize}{!}{$a_{n-1}$}};
				\put(-20,50){\resizebox{.06\hsize}{!}{$a_{n}$}};
				\put(50,-20){\resizebox{.06\hsize}{!}{$(b)$}};	
				\put(25,20){\resizebox{.2\hsize}{!}{$\frac{m-k_{n}\cdot(2\alpha-1)}{2}$}};
				\put(30,60){\resizebox{.11\hsize}{!}{$2\alpha+1$}};
				\put(150,15){\resizebox{.11\hsize}{!}{$2\alpha+1$}};
				\end{tikzpicture}
				
			\end{minipage}
			\vspace{0.5 cm}
			\caption{The catenary graph of $\beta_{3}=2m$} \label{fig:2 }
		\end{figure}

	\end{itemize}	
\end{proof}
\begin{corollary}
	Let $\mathbb{S}$ be a member of the telescopic numerical semigroup family $\Phi$ given in Theorem \ref{thm1}. The catenary degree of $\mathbb{S}$ is following that:
	$$	c(\mathbb{S})=\max\left( 2\alpha+1,\frac{m-\max\{k\}\cdot(2\alpha-1)}{2}\right) $$
	for $k\leq\frac{m}{2\alpha+1}$ and $k\in \mathbb{N}_{o}$.
\end{corollary}	
\begin{exam}
	Let $\mathbb{S}=\langle4,10,23\rangle \in\Phi$. Then $\beta_{1}=\beta_{2}=20$ and  $\beta_{3}=46$. The factorizations of $\beta_{1}=\beta_{2}=20$ are  $(5,0,0)$ and $(0,0,2)$; the factorizations of $\beta_{3}=46$ are  $(9,1,0)$, $(4,3,0)$ and $(0,0,2)$. However,  $c(\beta_{1})=c(\beta_{2})=c(20)=5$ and $c(\beta_{3})=c(46)=7$. Thus, $c(\mathbb{S})=7$.
\end{exam}
\begin{theorem}\label{thm7}
	Let $\mathbb{S}$ be a member of the telescopic numerical semigroup family $\prod$ given in Theorem \ref{thm2} such that $\beta_{1},\beta_{2}$  and $\beta_{3}$ be Betti elements of the numerical semigroup $\mathbb{S}$. In this case,
	\begin{itemize}
		\item[i)]If $6\alpha+2<m\leq9\alpha+3$ and $3|m$, then the factorizations of $\beta_{1}=2m$ are $(\frac{m-k\cdot(3\alpha+1)}{3},k,0)$ and $(0,0,2)$ for $ k\leq \frac{m}{3\alpha+1}$, $3| m-k\cdot(3\alpha+1)$ and $k\in\mathbb{N}$. If other, then the factorizations of $\beta_{1}=18\alpha+6$ are $(3\alpha+1,0,0)$ and $(0,3,0)$.
		\item[ii)]If $6\alpha+2<m\leq9\alpha+3$ and $3|m$, then the factorizations of $\beta_{2}=18\alpha+6$ are $(\frac{9\alpha+3-m}{3},0,2)$,
		$(3\alpha+1,0,0)$ and $(0,3,0)$. If other, then the factorizations of $\beta_{2}=18\alpha+6$ are $(3\alpha+1,0,0)$ and $(0,3,0)$.
		\item[iii)] The factorizations of $\beta_{3}=2m$ are $(\frac{m-k\cdot(3\alpha+1)}{3},k,0)$ and $(0,0,2)$ for $ k\leq \frac{m}{3\alpha+1}$, $3| m-k\cdot(3\alpha+1)$ and $k\in\mathbb{N}$.
	\end{itemize}	
\end{theorem}
\begin{proof} The Betti elements of the numerical semigroup $\mathbb{S}$ given in Theorem \ref{thm2} are  $\beta_{1},\beta_{2}$  and $\beta_{3}$ in the proof of Theorem \ref{thm4}. According to Theorem \ref{thm4},
	
	\begin{itemize}
		\item[i)] Firstly, we will find the factorizations of $\beta_{1}$.
		\begin{itemize}
			\item[a)] If $6\alpha+2<m\leq9\alpha+3$ and $3|m$, then the factorizations of $\beta_{1}=2m$. We write $\beta_{1}=2m=6x_{1}+(6\alpha+2)x_{2}+mx_{3} \quad(x_{1},x_{2},x_{3}\in\mathbb{N})$. In this case, it clear that $x_{3}$ is one of $0,1$ or $2$. If $x_{3}=0$, then $x_{1}=\frac{m-(3\alpha+1)\cdot x_{2}}{3}$, $x_{2}=\frac{m-3x_{1}}{3\alpha+1}$ and $x_{1},x_{2}\in\mathbb{N}$. Assume that  $x_{2}=k\in\mathbb{N}$, then it must be $3| m-k\cdot(3\alpha+1)$ and $k\leq\frac{m}{3\alpha+1}$. Thus, if $x_{3}=0$, then the factorization of $\beta_{1}=2m$ is $(\frac{m-k\cdot(3\alpha+1)}{3},k,0)$. If $x_{3}=1$, then the equation $m=6x_{1}+(6\alpha+2)x_{2}$ is obtained. But this contradicts that $m$ is a positive odd integer. If $x_{3}=2$, then we write $0=6x_{1}+(6\alpha+2)x_{2}$. Hence, it is clear that $x_{1}=0$ and $x_{2}=0$. Thus, the factorization of $\beta_{1}=2m$ is $(0,0,2)$. 
			\item[b)] In other cases, $\beta_{1}=18\alpha+6$. We can write $\beta_{1}=18\alpha+6=6x_{1}+(6\alpha+2)x_{2}+mx_{3} \quad(x_{1},x_{2},x_{3}\in\mathbb{N})$. Since $18\alpha+6$ is a positive even  integer, $x_{3}$ must be a nonnegative even integer. Furthermore, since $m>6\alpha+2$, it should be $x_{3}=0$ or $x_{3}=2$. If $x_{3}=0$, then $x_{2}=3-\frac{3x_{1}}{3\alpha+1}$. So, $x_{1}$ must be  $0$ or $3\alpha+1$. Therefore, if $x_{3}=0$, then the factorizations of $\beta_{1}=18\alpha+6$ are $(3\alpha+1,0,0)$ and $(0,3,0)$. If $x_{3}=2$, then $x_{2}=3-\frac{3x_{1}+m}{3\alpha+1}$. Thus, the fraction $\frac{3x_{1}+m}{3\alpha+1}$ is one of $0,1,2,$ or $3$. If $\frac{3x_{1}+m}{3\alpha+1}=0$, then $m=-3x_{1}$. But this statement contradicts the acceptance of $x_{1}$  and $m$.  If $\frac{3x_{1}+m}{3\alpha+1}=1$, then $x_{1}=3-\frac{(3\alpha+1)-m}{3}$. But, since $m>6\alpha+2$, this statement contradicts the acceptance of $x_{1}$. If $\frac{3x_{1}+m}{3\alpha+1}=2$ and $\frac{3x_{1}+m}{3\alpha+1}=3$, then a similar contradictions are obtained. 
		\end{itemize}
		\item[ii)] Now, we will find the factorizations of $\beta_{2}=18\alpha+6$. We can write $\beta_{2}=18\alpha+6=6x_{1}+(6\alpha+2)x_{2}+mx_{3} \quad(x_{1},x_{2},x_{3}\in\mathbb{N})$. Since $18\alpha+6$ is a positive even  integer, $x_{3}$ must be a nonnegative even integer. Since $m>6\alpha+2$, it should be $x_{3}=0$ or $x_{3}=2$. If $x_{3}=0$, then the factorizations of $\beta_{2}=18\alpha+6$ are obtained as $(3\alpha+1,0,0)$ and $(0,3,0)$ as in item i)-b). If $x_{3}=2$, then $x_{2}=3-\frac{3x_{1}+m}{3\alpha+1}$. And, the fraction $\frac{3x_{1}+m}{3\alpha+1}$ is one of $0,1,2,$ or $3$. When the fraction $\frac{3x_{1}+m}{3\alpha+1}$ is one of $0,1,$ or $2$, there are similar contradictions as in item i)-b). But, if $\frac{3x_{1}+m}{3\alpha+1}=3$, then $x_{1}=\frac{9\alpha+3-m}{3}$. So $x_{1}$ is a nonnegative integer if and only if $6\alpha+2<m\leq9\alpha+3$ and $3|m$. Under these conditions, the factorization of $\beta_{2}=18\alpha+6$ is $(\frac{9\alpha+3-m}{3},0,2)$
		\item[iii)] Let’s find the factorizations of $\beta_{3}=2m$. We write $\beta_{3}=2m=6x_{1}+(6\alpha+2)x_{2}+mx_{3} \quad(x_{1},x_{2},x_{3}\in\mathbb{N})$. In this  case, it is clear that $x_{3}$ is one of $0,1$ or $2$. If $x_{3}=0$, then $x_{1}=\frac{m-(3\alpha+1)\cdot x_{2}}{3}$ and $x_{2}=\frac{m-3x_{1}}{3\alpha+1}$. Assume that $x_{2}=k\in\mathbb{N}$, then $3| m-k\cdot(3\alpha+1)$ and $k\leq\frac{m}{3\alpha+1}$. Thus, If $x_{3}=0$, then the factorization of $\beta_{3}=2m$ is $(\frac{m-k\cdot(3\alpha+1)}{3},k,0)$. If $x_{3}=1$, then the equation $m=6x_{1}+(6\alpha+2)x_{2}$ is obtained. But this contradicts that $m$ is a positive odd integer. If $x_{3}=2$, then we write $0=6x_{1}+(6\alpha+2)x_{2}$. Hence, it is clear that $x_{1}=0$ and $x_{2}=0$. Thus, the factorization of $\beta_{3}=2m$ is $(0,0,2)$.
	\end{itemize}		 
\end{proof}
\begin{theorem}\label{thm8} Let $\mathbb{S}$ be a member of the telescopic numerical semigroup family $\prod$ given in Theorem \ref{thm2}. The catenary degrees of Betti elements of $\mathbb{S}$ are as follows:		
	\begin{itemize}
		\item[i)] \begin{displaymath}
		c(\beta_{1}) = \left\{ \begin{array}{ll}
		\dfrac{m}{3} & \textrm{if } 6\alpha+2<m\leq9\alpha+3 \textrm{ and } 3|m \\
		3\alpha+1 & \textrm{if other}
		\end{array} \right.
		\end{displaymath}
		\item[ii)] \begin{displaymath}
		c(\beta_{2}) =\left\{ \begin{array}{ll}
		\dfrac{m}{3} & \textrm{if } 6\alpha+2<m\leq9\alpha+3 \textrm{ and } 3|m \\
		3\alpha+1 & \textrm{if other}
		\end{array} \right.
		\end{displaymath}
		\item[iii)] $c(\beta_{3})=\max\{3\alpha+1,\frac{m-\max\{k\}\cdot (3\alpha-2)}{3}\} $  for $3| m-k(3\alpha+1)$, $k\leq\frac{m}{3\alpha+1}$ and $k\in\mathbb{N}$.
	\end{itemize}	
\end{theorem}
\begin{proof} Assume that $\mathbb{S}$ is a member of the telescopic numerical semigroup family $\prod$ given in Theorem \ref{thm2}. From the proof of Theorem \ref{thm4}, we know the Betti elements of $\mathbb{S}$. Moreover, the factorizations of the Betti elements of the numerical semigroup $\mathbb{S}$ are given in Theorem \ref{thm7}.
	\begin{itemize}
		\item[i)] Firstly, we will find the catenary degree of $\beta_{1}$.
		\begin{itemize}
			\item[a)] From Theorem \ref{thm7}, if $6\alpha+2<m\leq9\alpha+3$ and $3|m$, then the factorizations of $\beta_{1}=2m$ are $(0,0,2)$ and $(\frac{m-k\cdot(3\alpha+1)}{3},k,0)$ for $k\leq \frac{m}{3\alpha+1}$, $3| m-k\cdot(3\alpha+1)$ and $k\in\mathbb{N}$. If $6\alpha+2<m<9\alpha+3$ and $3|m$, then $k=0$. And the factorizations of $\beta_{1}=2m$ are $(0,0,2)$ and $(\dfrac{m}{3},0,0)$. If $m=9\alpha+3$, then $k=0$ or $k=3$. And the factorizations of $\beta_{1}=2m$ are $(0,0,2)$, $(\dfrac{m}{3},0,0)$ and $(0,3,0)$. Thus, the distances of the edge between these factorizations are found as follows:
			$$\gcd\{(0,0,2),(\dfrac{m}{3},0,0)\}=(0,0,0)$$
			$$\gcd\{(0,0,2),(0,3,0)\}=(0,0,0)$$
			$$\gcd\{(0,3,0),(\dfrac{m}{3},0,0)\}=(0,0,0)$$
			and
			$$dist\{(0,0,2),(\dfrac{m}{3},0,0)\}=\dfrac{m}{3}$$
			$$dist\{(0,0,2),(0,3,0)\}=3$$
			$$dist\{(0,3,0),(\dfrac{m}{3},0,0)\}=\dfrac{m}{3}$$	
			When each vertex is labeled with one of the factorizations of $\beta_{1}=2m$ and each edge is labeled with distance between the factorizations of $\beta_{1}=2m$ at either end, we get Figure \ref{fig:3} and  Figure \ref{fig:4}. Hence, if we draw the graphs in Figure \ref{fig:3} and  Figure \ref{fig:4} which consist of these vertices and edges, then the catenary degree of $\beta_{1}=2m$ is $\dfrac{m}{3}$.\\ 	
			When $6\alpha+2<m\leq9\alpha+3$ and $3|m$, we get Figure \ref{fig:3}
			\\\indent
			\begin{figure}[H]
				\centering
				\begin{tikzpicture}
				\draw[thick,-] (0,0)--(10,0);
				\draw[thick,-] (10,0)--(0,0);
				\filldraw[black] (0,0) circle (2pt);
				\filldraw[black] (10,0) circle (2pt);
				\put(-20,5){$(0,0,2)$};
				\put(270,10){$(\dfrac{m}{3},0,0)$};
				\put(150,10){$\dfrac{m}{3}$};
				\end{tikzpicture}
				\vspace{0.5 cm}
				\caption{The catenary graph of $\beta_{1}=2m$ with factorizations $(0,0,2)$ and $(\dfrac{m}{3},0,0)$ } \label{fig:3}
			\end{figure}
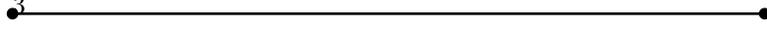
			When $m=9\alpha+3$, we get Figure \ref{fig:4}
			\\
			\begin{figure}[H]
				\centering
				\begin{minipage}{.5\textwidth}
					\centering
					\centering
					\begin{tikzpicture}[thick,scale=0.47, every node/.style={scale=0.6}]
					\draw[thick,-] (0,0)--(10,0);
					\draw[thick,-] (10,0)--(5,5);
					\draw[thick,-] (5,5)--(0,0);
					\filldraw[black] (0,0) circle (2pt);
					\filldraw[black] (10,0) circle (2pt);
					\filldraw[black] (5,5) circle (2pt);
					\put(-20,-8){\resizebox{.09\hsize}{!}{$(0,3,0)$}};
					\put(120,-8){\resizebox{.09\hsize}{!}{$(0,0,2)$}};
					\put(50,5){\resizebox{.02\hsize}{!}{$3$ }};
					\put(50,75){\resizebox{.15\hsize}{!}{$(\dfrac{m}{3},0,0)$}};
					\put(20,35){\resizebox{.04\hsize}{!}{$\frac{m}{3}$}};
					\put(100,35){\resizebox{.04\hsize}{!}{$\frac{m}{3}$}};
					\put(50,-20){\resizebox{.04\hsize}{!}{$(a)$}};
					\end{tikzpicture}
				\end{minipage}%
				\begin{minipage}{0.5\textwidth}
					
					\begin{tikzpicture}[thick,scale=0.47, every node/.style={scale=0.6}]
					\draw[thick,-] (0,0)--(0,0);
					\draw[thick,-] (0,0)--(5,5);
					\draw[thick,-] (5,5)--(10,0);
					\filldraw[black] (0,0) circle (2pt);
					\filldraw[black] (10,0) circle (2pt);
					\filldraw[black] (5,5) circle (2pt);
					\put(-20,-12){\resizebox{.09\hsize}{!}{$(0,3,0)$}};
					\put(120,-8){\resizebox{.09\hsize}{!}{$(0,0,2)$}};
					\put(50,75){\resizebox{.15\hsize}{!}{$(\dfrac{m}{3},0,0)$}};
					\put(20,35){\resizebox{.04\hsize}{!}{$\frac{m}{3}$}};
					\put(100,35){\resizebox{.04\hsize}{!}{$\frac{m}{3}$}};
					\put(50,-20){\resizebox{.04\hsize}{!}{$(b)$}};
					\end{tikzpicture}
				\end{minipage}
				\vspace{0.5 cm}
				\caption{The catenary graph of $\beta_{1}=2m$ with factorizations $(0,0,2)$, $(0,3,0)$ and $(\dfrac{m}{3},0,0)$ } \label{fig:4}
			\end{figure}
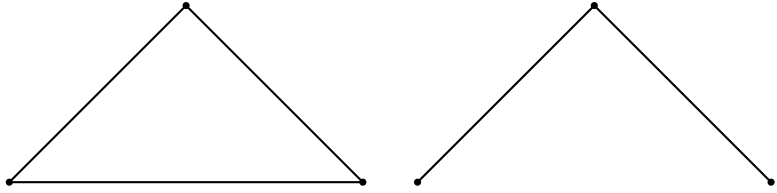

			\item[b)] From Theorem \ref{thm7}, we know that the factorizations of $\beta_{1}=18\alpha+6$ are $(3\alpha+1,0,0)$ and $(0,3,0)$ if other cases. In this case, the distance of the edge between these factorizations is found as
			$$\gcd\{(3\alpha+1,0,0),(0,3,0)\} = (0,0,0)$$
			$$dist\{(3\alpha+1,0,0),(0,3,0)\}=3\alpha+1.$$	 	
			Hence, if we draw the graph in Figure \ref{fig:5} which consist of these vertice and edge, then the catenary degree of $\beta_{1}=18\alpha+6$ is $3\alpha+1$.
			\\
			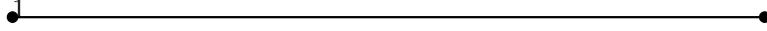
\begin{figure}[H]
				\centering
				\begin{tikzpicture}
				\draw[thick,-] (0,0)--(10,0);
				\draw[thick,-] (10,0)--(0,0);
				\filldraw[black] (0,0) circle (2pt);
				\filldraw[black] (10,0) circle (2pt);
				\put(-20,5){$(0,3,0)$};
				\put(270,10){$(3\alpha+1,0,0)$};
				\put(150,10){$3\alpha+1$};
				\end{tikzpicture}
				\vspace{0.5 cm}
				\caption{The catenary graph of $\beta_{1}=18\alpha+6$ with factorizations $(0,3,0)$ and $(3\alpha+1,0,0)$ } \label{fig:5}
			\end{figure}
		\end{itemize}	
		\item[ii)] Now, we will find the catenary degree of $\beta_{2}$.
		\begin{itemize}
			\item[a)]	From Theorem \ref{thm7}, if $6\alpha+2<m\leq9\alpha+3$ and $3|m$, then the factorizations of $\beta_{2}=18\alpha+6$ are $(\frac{9\alpha+3-m}{3},0,2)$, $(3\alpha+1,0,0)$ and $(0,3,0)$. In this case, the distances of the edges between these factorizations are found as follows:
			$$\gcd\{(\frac{9\alpha+3-m}{3},0,2),(3\alpha+1,0,0)\}=(0,0,0)$$
			$$\gcd\{(0,3,0),(3\alpha+1,0,0)\}=(0,0,0)$$
			$$\gcd\{(\frac{9\alpha+3-m}{3},0,2),(0,3,0)\}=(0,0,0)$$
			and
			$$dist\{(\frac{9\alpha+3-m}{3},0,2),(3\alpha+1,0,0)\}=\dfrac{m}{3}$$
			$$dist\{(0,3,0),(3\alpha+1,0,0)\}=3\alpha+1$$
			$$dist\{(\frac{9\alpha+3-m}{3},0,2),(0,3,0\}=3+3\alpha-\dfrac{m}{3}$$	
			
			When each vertex is labeled with one of the factorizations of $\beta_{2}=18\alpha+6$ and each edge is labeled with distance between the factorizations of $\beta_{2}=18\alpha+6$ at either end, we get  Figure \ref{fig:6}.	Thus, if we draw the graph in Figure \ref{fig:6}, which consist of these vertices and edges, then the catenary degrees of $\beta_{2}=18\alpha+6$ is $\dfrac{m}{3}$. Because $3\alpha+1>\dfrac{m}{3}\geq3+3\alpha-\dfrac{m}{3}$ for all $\alpha\in \mathbb{N}$ and $m\in \mathbb{N}_{o}$.
			\\
			\begin{figure}[H]
				\centering
				\begin{minipage}{.5\textwidth}
					\centering
					\centering
					\begin{tikzpicture}[thick,scale=0.47, every node/.style={scale=0.6}]
					\draw[thick,-] (0,0)--(10,0);
					\draw[thick,-] (10,0)--(5,5);
					\draw[thick,-] (5,5)--(0,0);
					\filldraw[black] (0,0) circle (2pt);
					\filldraw[black] (10,0) circle (2pt);
					\filldraw[black] (5,5) circle (2pt);
					\put(-10,-8){\resizebox{.15\hsize}{!}{$(3\alpha+1,0,0)$}};
					\put(120,-8){\resizebox{.09\hsize}{!}{$(0,3,0)$}};
					\put(50,5){\resizebox{.10\hsize}{!}{$3\alpha+1$ }};
					\put(50,75){\resizebox{.19\hsize}{!}{$(\frac{9\alpha+3-m}{3},0,2)$}};
					\put(20,35){\resizebox{.04\hsize}{!}{$\frac{m}{3}$}};
					\put(100,35){\resizebox{.15\hsize}{!}{$3+3\alpha-\dfrac{m}{3}$}};
					\put(50,-20){\resizebox{.04\hsize}{!}{$(a)$}};
					\end{tikzpicture}
				\end{minipage}%
				\begin{minipage}{0.5\textwidth}
					
					\begin{tikzpicture}[thick,scale=0.47, every node/.style={scale=0.6}]
					\draw[thick,-] (0,0)--(0,0);
					\draw[thick,-] (0,0)--(5,5);
					\draw[thick,-] (5,5)--(10,0);
					\filldraw[black] (0,0) circle (2pt);
					\filldraw[black] (10,0) circle (2pt);
					\filldraw[black] (5,5) circle (2pt);
					\put(-20,-12){\resizebox{.15\hsize}{!}{$(3\alpha+1,0,0)$}};
					\put(120,-8){\resizebox{.09\hsize}{!}{$(0,3,0)$}};
					\put(50,75){\resizebox{.19\hsize}{!}{$(\frac{9\alpha+3-m}{3},0,2)$}};
					\put(20,35){\resizebox{.04\hsize}{!}{$\frac{m}{3}$}};
					\put(100,35){\resizebox{.15\hsize}{!}{$3+3\alpha-\dfrac{m}{3}$}};
					\put(50,-20){\resizebox{.04\hsize}{!}{$(b)$}};
					\end{tikzpicture}
				\end{minipage}
				\vspace{0.5 cm}
				\caption{The catenary graph of $\beta_{2}=18\alpha+6$ with factorizations $(\frac{9\alpha+3-m}{3},0,2)$, $(3\alpha+1,0,0)$ and $(0,3,0)$} \label{fig:6}
			\end{figure}
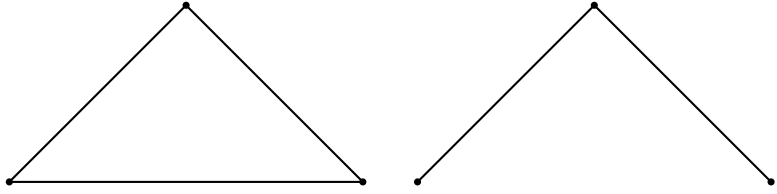
			
			\item[b)] From Theorem \ref{thm7}, we know that the factorizations of $\beta_{2}=18\alpha+6$ are $(3\alpha+1,0,0)$ and $(0,3,0)$ if other cases. The proof is as in item i)-b). Namely, the catenary degrees of $\beta_{2}=18\alpha+6$ is $3\alpha+1$.		
		\end{itemize}
		\item[iii)] Finally, we will find the catenary degree of $\beta_{3}$. From Theorem \ref{thm7},	the factorizations of $\beta_{3}=2m$ are  $(0,0,2)$ and $(\frac{m-k\cdot(3\alpha+1)}{3},k,0)$ for $ k\leq \frac{m}{3\alpha+1}$, $3| m-k\cdot(3\alpha+1)$ and $k\in\mathbb{N}$. In this case, the distances of the edges between these factorizations are found as follows: \\ 	\indent 	Since there will be an edge for every nonnegative integer $k$, let's show that the edge corresponding to each $k_{i}$ with $a_{i}$ such that $a_{i}=(\frac{m-k_{i}\cdot(3\alpha+1)}{3},k_{i},0)$ for $i\in \left\lbrace 1,2,\dots,n\right\rbrace $. Where $k_{1}<k_{2}<\dots<k_{n}$
		$$\gcd\{a_{i},(0,0,2)\}=(0,0,0)$$
		and
		$$dist\{a_{i},(0,0,2)\}=\frac{m-k_{i}\cdot(3\alpha+1)}{3}+k_{i}=\frac{m-k_{i}\cdot(3\alpha-2)}{3}$$
		Let $i\in \left\lbrace 1,2,\dots,n-1\right\rbrace $ and $j\in \left\lbrace 2,3,\dots,n\right\rbrace $ such that $i<j $.
		
		$$\gcd\{a_{i},a_{j}\}=(\frac{m-k_{j}\cdot(3\alpha+1)}{3},k_{i},0)$$
		and
		$$dist\{a_{i},a_{j}\}=\max\{|\frac{(k_{j}-k_{i})\cdot(3\alpha+1)}{3}|,|k_{j}-k_{i}|\}$$
		It is clear that $k_{j}=k_{1}+3(j-1)$ and  $|\frac{(k_{j}-k_{i})(3\alpha+1)}{3}|>|k_{j}-k_{i}|$. Thus,
		$$dist\{a_{i},a_{j}\}=\frac{(k_{j}-k_{i})\cdot(3\alpha+1)}{3}.$$
		
		The following equations can  easily be seen from these obtained above:
		\begin{multline*}
		\min\{dist\{a_{1},(0,0,2)\},dist\{a_{2},(0,0,2)\},\dots,dist\{a_{n},(0,0,2)\}\} = \min\{\frac{m-k_{1}\cdot(3\alpha-2)}{3}, \\
		\frac{m-k_{2}\cdot(3\alpha-2)}{3},\dots,\frac{m-k_{n}\cdot(3\alpha-2)}{3}\}=\frac{m-k_{n}\cdot(3\alpha-2)}{3}=dist\{a_{n},(0,0,2)\}
		\end{multline*}
		
		for $i\in \left\lbrace 1,2,\dots,n\right\rbrace $
		\begin{multline*}
		\min\{dist\{a_{i},a_{1}\},dist\{a_{i},a_{2}\},\dots,dist\{a_{i},a_{i-1}\},dist\{a_{i},a_{i+1}\},\dots,dist\{a_{i},a_{n}\}\}= \\ \min\{|\frac{(k_{i}-k_{1})\cdot(3\alpha+1)}{3}|,|\frac{(k_{i}-k_{2})\cdot(3\alpha+1)}{3}|,\\ \dots,|\frac{(k_{i}-k_{i-1})\cdot(3\alpha+1)}{3}|,|\frac{(k_{i+1}-k_{i})\cdot(3\alpha+1)}{3}|,\\
		\dots,|\frac{(k_{n}-k_{i})\cdot(3\alpha+1)}{3}|\} =|\frac{(k_{i}-k_{i-1})\cdot(3\alpha+1)}{3}|\\
		=|\frac{(k_{i+1}-k_{i})\cdot(3\alpha+1)}{3}|=\frac{[(k_{1}+(i-1)3)-(k_{1}+(i-1-1)3)]\cdot(3\alpha+1)}{3}\\
		=\frac{[(k_{1}+(i+1-1)3)-(k_{1}+(i-1)3)]\cdot(3\alpha+1)}{3}=3\alpha+1
		\end{multline*}
		When each vertex is labeled with one of the factorizations of $\beta_{3}=2m$ and each edge is labeled with distance between the factorizations of $\beta_{3}=2m$ at either end, we get Figure \ref{fig:7 }$(a)$. If vertices with maximal length are removed from the connected graph in Figure \ref{fig:7 }$(a)$, then Figure\ref{fig:7 }$(b)$ is obtained. So, the catenary degree of $\beta_{3}=2m$ is $\max\left( 3\alpha+1,\frac{m-k_{n}(3\alpha-2)}{3}\right) $, where $k_{n}=\max\{k\}$ for $3| m-k(3\alpha+1)$, $k\leq\frac{m}{3\alpha+1}$ and $k\in\mathbb{N}$.
		\\
		\begin{figure}[H]
			\centering
			\begin{minipage}{.5\textwidth}
				\centering
				\centering
				\begin{tikzpicture}[thick,scale=0.47, every node/.style={scale=0.6}]
				\draw[thick,-] (0,4)--(3,0);
				\draw[thick,-] (3,0)--(8,0);
				\begin{scope}[very thick,dashed]
				\draw[thick,-] (8,0)--(13,3);
				\end{scope}
				\draw[thick,-] (13,3)--(10,7);
				\draw[thick,-] (10,7)--(5,7);
				\begin{scope}[very thick,dashed]
				\draw[thick,-] (5,7)--(0,4);
				\end{scope}
				\draw[thick,-] (0,4)--(8,0);
				\draw[thick,-] (0,4)--(13,3);
				\draw[thick,-] (0,4)--(10,7);
				\draw[thick,-] (3,0)--(13,3);
				\draw[thick,-] (3,0)--(10,7);
				\draw[thick,-] (3,0)--(5,7);
				\draw[thick,-] (8,0)--(10,7);
				\draw[thick,-] (8,0)--(5,7);
				\draw[thick,-] (13,3)--(5,7);
				\filldraw[black] (0,4) circle (2pt);
				\filldraw[black] (3,0) circle (2pt);
				\filldraw[black] (8,0) circle (2pt);
				\filldraw[black] (13,3) circle (2pt);
				\filldraw[black] (10,7) circle (2pt);
				\filldraw[black] (5,7) circle (2pt);
				\put(10,-10){\resizebox{.12\hsize}{!}{$(0,0,2)$}};
				\put(100,-10){\resizebox{.06\hsize}{!}{$a_{1}$}};
				\put(170,30){\resizebox{.09\hsize}{!}{$a_{i-1}$}};
				\put(150,90){\resizebox{.06\hsize}{!}{$a_{i}$}};
				\put(40,95){\resizebox{.09\hsize}{!}{$a_{i+1}$}};
				\put(-20,50){\resizebox{.06\hsize}{!}{$a_{n}$}};
				\put(50,-20){\resizebox{.06\hsize}{!}{$(a)$}};	
				\end{tikzpicture}
			\end{minipage}%
			\begin{minipage}{0.5\textwidth}
				
				\begin{tikzpicture}[thick,scale=0.47, every node/.style={scale=0.6}]
				\draw[thick,-] (0,4)--(3,0);
				\draw[thick,-] (8,0)--(13,3);
				\begin{scope}[very thick,dashed]
				\draw[thick,-] (13,3)--(10,7);
				\end{scope}
				\begin{scope}[very thick,dashed]
				\draw[thick,-] (10,7)--(5,7);
				\end{scope}
				\draw[thick,-] (5,7)--(0,4);
				\filldraw[black] (0,4) circle (2pt);
				\filldraw[black] (3,0) circle (2pt);
				\filldraw[black] (8,0) circle (2pt);
				\filldraw[black] (13,3) circle (2pt);
				\filldraw[black] (10,7) circle (2pt);
				\filldraw[black] (5,7) circle (2pt);
				\put(10,-10){\resizebox{.12\hsize}{!}{$(0,0,2)$}};
				\put(100,-10){\resizebox{.06\hsize}{!}{$a_{1}$}};
				\put(170,30){\resizebox{.06\hsize}{!}{$a_{2}$}};
				\put(150,90){\resizebox{.06\hsize}{!}{$a_{i}$}};
				\put(40,95){\resizebox{.09\hsize}{!}{$a_{n-1}$}};
				\put(-20,50){\resizebox{.06\hsize}{!}{$a_{n}$}};
				\put(50,-20){\resizebox{.06\hsize}{!}{$(b)$}};	
				\put(25,20){\resizebox{.2\hsize}{!}{$\frac{m-k_{n}\cdot(3\alpha-2)}{3}$}};
				\put(30,60){\resizebox{.11\hsize}{!}{$3\alpha+1$}};
				\put(150,15){\resizebox{.11\hsize}{!}{$3\alpha+1$}};
				\end{tikzpicture}
				
			\end{minipage}
			\vspace{0.5 cm}
			\caption{The catenary graph of $\beta_{3}=2m$} \label{fig:7 }
		\end{figure}
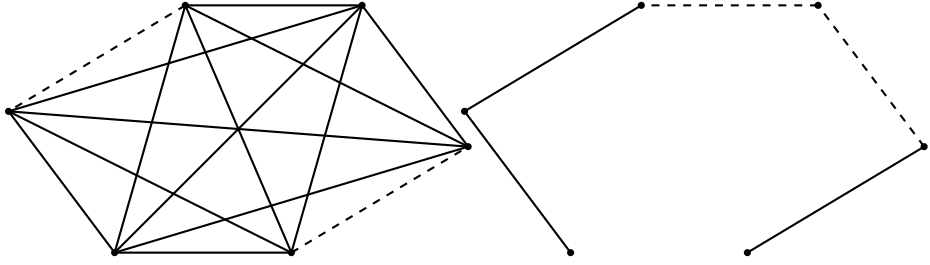
		
	\end{itemize}
\end{proof}		
\begin{corollary}
	Let $\mathbb{S}$ be a member of the telescopic numerical semigroup family $\prod$ given in Theorem \ref{thm2}. The catenary degrees of $\mathbb{S}$ are as follows:		
	$$	c(\mathbb{S})=\max\{3\alpha+1,\frac{m-\max\{k\}\cdot (3\alpha-2)}{3}\}$$
	for $3| m-k(3\alpha+1)$, $k\leq\frac{m}{3\alpha+1}$ and $k\in\mathbb{N}$.		
\end{corollary}	
\begin{exam}
	Let $\mathbb{S}=\langle 6,14,21\rangle \in\prod$. Then $\beta_{1}=\beta_{2}=\beta_{3}=42$. The factorizations of $\beta_{1}=\beta_{2}=\beta_{3}=42$ are  $(7,0,0)$, $(0,0,3)$ and $(0,0,2)$. Also,  $c(\beta_{1})=c(\beta_{2})=c(\beta_{3})=7$. Thus, $c(\mathbb{S})=7$.
\end{exam}
\begin{theorem}\label{thm9}
	Let $\mathbb{S}$ be a member of the telescopic numerical semigroup family $\Omega$ given in Theorem \ref{thm2} such that $\beta_{1}=\beta_{2}=12\alpha+6$  and $\beta_{3}=2n$ be Betti elements of the numerical semigroup $\mathbb{S}$. In this case, the factorizations of $\beta_{1}=\beta_{2}=12\alpha+6$ are $(2\alpha+1,0,0)$ and $(0,2,0)$; the factorizations of $\beta_{3}=3n$ are $(\frac{n-k\cdot(2\alpha+1)}{2},k,0)$ and $(0,0,3)$ for $k\leq\frac{n}{2\alpha+1}$ and $2|n-k\cdot(2\alpha+1)$ for $k\in\mathbb{N}$.
\end{theorem}
\begin{proof}The Betti elements of the numerical semigroup $\mathbb{S}$, which is a member of the telescopic numerical semigroup family $\Omega$ given in Theorem \ref{thm2}, are $\beta_{1}=\beta_{2}=12\alpha+6$  and $\beta_{3}=3n$ in the proof of Theorem \ref{thm4}. Let’s find the factorizations of $\beta_{1}=\beta_{2}=12\alpha+6$. We write $\beta_{1}=\beta_{2}=12\alpha+6=6x_{1}+(6\alpha+3)x_{2}+nx_{3} \quad(x_{1},x_{2},x_{3}\in\mathbb{N})$. In this case, it should be $x_{3}=0$ or $x_{3}=1$ since $x_{3}>6\alpha+3$. If $x_{3}=0$, then the equality $12\alpha+6=6x_{1}+(6\alpha+3)x_{2}$ is obtained. Hence, $x_{2}=2-\frac{2x_{1}}{2\alpha+1}$. And since $x_{1},x_{2}\in\mathbb{N}$, $x_{1}=0$ or $x_{1}=2\alpha+1$. Therefore, if $x_{3}=0$, then the factorizations of $\beta_{1}=\beta_{2}=12\alpha+6$ are $(2\alpha+1,0,0)$ and $(0,2,0)$. If $x_{3}=1$, then the equality $12\alpha+6=6x_{1}+(6\alpha+3)x_{2}+n$ is obtained. In this case, $x_{2}=2-\frac{6x_{1}+n}{6\alpha+3}$. And since $x_{1},x_{2}\in\mathbb{N}$, it is obtained that $(6\alpha+3)|(6x_{1}+n)$. But this contradicts $3\nmid n$. Thus, any factorizations of $\beta_{1}=\beta_{2}=12\alpha+6$ can not be found for $x_{3}=1$.
	
	Now let’s find the factorizations of $\beta_{3}=3n$. We write $\beta_{3}=3n=6x_{1}+(6\alpha+3)x_{2}+nx_{3} \quad(x_{1},x_{2},x_{3}\in\mathbb{N})$. In this  case, $x_{3}$ is equal to $0, 1, 2,$  or $3$. If $x_{3}=0$, then the equality $3n=6x_{1}+(6\alpha+3)x_{2}$. Hence, $x_{2}=\frac{n-2x_{1}}{2\alpha+1}$ and $x_{1}=\frac{n-k\cdot(2\alpha+1)}{2}$. Since $x_{1},x_{2}\in\mathbb{N}$, $x_{2}\leq\frac{n}{2\alpha+1}$ and $2|n-k\cdot(2\alpha+1)$ for $x_{2}=k\in\mathbb{N}$. Thus, one of the factorizations of $\beta_{3}=3n$ is $(\frac{n-k\cdot(2\alpha+1)}{2},k,0)$ for $k\in\mathbb{N}$ and  $k\leq\frac{n}{2\alpha+1}$. If $x_{3}=1$, then $2n=6x_{1}+(6\alpha+3)x_{2}$. Thus,  we write $n=3(x_{1}+\alpha x_{2}+\frac{x_{2}}{2})$. But this expression contradicts $3\nmid n$ for $n\in\mathbb{N}$. If $x_{3}=2$, then the equality $n=6x_{1}+(6\alpha+3)x_{2}$ is obtained. We can write that $n=3(2x_{1}+(2\alpha+1)x_{2})$. But this expression contradicts $3\nmid n$ for $n\in\mathbb{N}$, too. If $x_{3}=3$, then the equality $0=6x_{1}+(6\alpha+3)x_{2}$ is obtained. Hence, it is clear that  $x_{1}=x_{2}=0$. Thus, in the case of $x_{3}=3$, the other  of the factorizations of $\beta_{3}=3n$ is  $(0,0,3)$.
\end{proof}
\begin{theorem}\label{thm10} Let $\mathbb{S}$ be a member of the telescopic numerical semigroup family $\Omega$ given in Theorem \ref{thm2}. The catenary degree of Betti elements of $\mathbb{S}$ is following that:
	\begin{itemize}
		\item[i)] $c(\beta_{1})=c(\beta_{2})=c(12\alpha+6)=2\alpha+1$
		\item[ii)] $c(\beta_{3})=c(2n)=\max\{2\alpha+1,\frac{n-\max\{k\}n-k\cdot(2\alpha+1)(2\alpha-1)}{2}\}$ for $k\leq\frac{n}{2\alpha+1}$, $2|n-k\cdot(2\alpha+1)$ and $k\in\mathbb{N}$
	\end{itemize}
\end{theorem}
\begin{proof} Assume that $\mathbb{S}$ is a member of the telescopic numerical semigroup family $\Omega$ given in Theorem \ref{thm2}. From the proof of Theorem \ref{thm4}, we know the Betti element of the numerical semigroup $\mathbb{S}$. Moreover, the factorizations of the Betti elements of $\mathbb{S}$ is given in Theorem \ref{thm9}.
	\item[i)] we will find the catenary degree of $\beta_{1}=\beta_{2}=12\alpha+6$. The length of the edge between these factorizations is found as $$\gcd\{(2\alpha+1,0,0),(0,2,0)\}=(0,0,0),$$  $$dist\{(2\alpha+1,0,0),(0,2,0)\}=2\alpha+1.$$ The factorization points that we found are vertices and the path that are connecting these vertices are edges. Hence, if we draw the graph in Figure \ref{fig:8}, which consists of  these vertices and edge, the catenary degree of $\beta_{1}$ is $2\alpha+1$.
	\\
	\begin{figure}[H]
		\centering
		\begin{tikzpicture}
		\draw[thick,-] (0,0)--(10,0);
		\draw[thick,-] (10,0)--(0,0);
		\put(-20,5){$\left( 2\alpha+1, 0, 0 \right)$};
		\put(270,5){$\left( 0, 2, 0 \right)$};
		\put(150,5){$2\alpha +1$};
		\filldraw[black] (0,0) circle (2pt);
		\filldraw[black] (10,0) circle (2pt);
		\end{tikzpicture}
		\vspace{0.5 cm}
		\caption{The catenary graph of $\beta_{1}=\beta_{2}=12\alpha+6$} \label{fig:8}
	\end{figure}
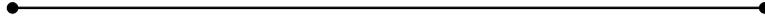
	
	\item[ii)] We will find  the catenary degree of $\beta_{3}=3n$. The factorizations of $\beta_{3}=3n$ are $(\frac{n-k\cdot(2\alpha+1)}{2},k,0)$ and $(0,0,3)$ for $k\leq\frac{n}{2\alpha+1}$ and $2|n-k\cdot(2\alpha+1)$ for $k\in\mathbb{N}$. In this case, the distances of the edges between these factorizations are found as follows: \\ 	\indent 	Since there will be an edge for every nonnegative integer $k$, let's show that the edge corresponding to each $k_{i}$ with $a_{i}$ such that $a_{i}=(\frac{n-k_{i}\cdot(2\alpha+1)}{2},k_{i},0)$ for $i\in \left\lbrace 1,2,\dots,n\right\rbrace $. Where $k_{1}<k_{2}<\dots<k_{n}$
	$$\gcd\{a_{i},(0,0,3)\}=(0,0,0)$$
	and
	$$dist\{a_{i},(0,0,3)\}=\frac{n-k_{i}\cdot(2\alpha+1)}{2}+k_{i}=\frac{n-k_{i}\cdot(2\alpha-1)}{2}$$
	Let $i\in \left\lbrace 1,2,\dots,n-1\right\rbrace $ and $j\in \left\lbrace 2,3,\dots,n\right\rbrace $ such that $i<j $.
	
	$$\gcd\{a_{i},a_{j}\}=(\frac{n-k_{j}(2\alpha+1)}{2},k_{i},0)$$
	and
	$$dist\{a_{i},a_{j}\}=\max\{|\frac{(k_{j}-k_{i})(2\alpha+1)}{2}|,|k_{j}-k_{i}|\}=\frac{(k_{j}-k_{i})(2\alpha+1)}{2}$$
	The following equations are resulted from those obtained above:  \\ 	\indent for $i\in \left\lbrace 1,2,\dots,n\right\rbrace $.
	\begin{multline*}
	\min\{dist\{a_{1},(0,0,3)\},dist\{a_{2},(0,0,3)\},\dots,dist\{a_{n},(0,0,3)\}\} = \min\{\frac{n-k_{1}\cdot(2\alpha-1)}{2}, \\
	\frac{n-k_{2}\cdot(2\alpha-1)}{2},\dots,\frac{n-k_{n}\cdot(2\alpha-1)}{2}\}=\frac{n-k_{n}\cdot(2\alpha-1)}{2}=dist\{a_{n},(0,0,3)\}
	\end{multline*}
	and
	\begin{multline*}
	\min\{dist\{a_{i},a_{1}\},dist\{a_{i},a_{2}\},\dots,dist\{a_{i},a_{i-1}\},dist\{a_{i},a_{i+1}\},\dots,dist\{a_{i},a_{n}\}\}= \\ \min\{|\frac{(k_{i}-k_{1})(2\alpha+1)}{2}|,|\frac{(k_{i}-k_{2})(2\alpha+1)}{2}|,\dots,|\frac{(k_{i}-k_{i-1})(2\alpha+1)}{2}|,|\frac{(k_{i}-k_{i+1})(2\alpha+1)}{2}|,\\
	\dots,|\frac{(k_{n}-k_{i})(2\alpha+1)}{2}|\} =|\frac{(k_{i}-k_{i-1})(2\alpha+1)}{2}|=|\frac{(k_{i+1}-k_{i})(2\alpha+1)}{2}|\\
	=2\alpha+1=dist\{a_{i},a_{i-1}\}=dist\{a_{i},a_{i+1}\}
	\end{multline*}
	When each vertex is labeled with one of the factorizations of $\beta_{3}=2n$ and each edge is labeled with distance between the factorizations of $\beta_{3}=2n$ at either end, we get Figure \ref{fig:9 } $(a)$. If vertices with maximal length are removed from the connected graph in Figure \ref{fig:9 } $(a)$, then Figure\ref{fig:9 } $(b)$ is obtained. Thus, the catenary degree of $\beta_{3}=2n$ is $\max\{(2\alpha+1,\frac{n-\max\{k\}(2\alpha-1)}{2})\} $.
	\\
	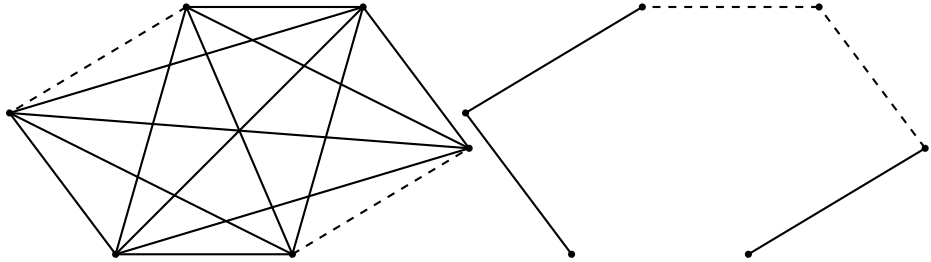
\begin{figure}[H]
		\centering
		\begin{minipage}{.5\textwidth}
			\centering
			\centering
			\begin{tikzpicture}[thick,scale=0.47, every node/.style={scale=0.6}]
			\draw[thick,-] (0,4)--(3,0);
			\draw[thick,-] (3,0)--(8,0);
			\begin{scope}[very thick,dashed]
			\draw[thick,-] (8,0)--(13,3);
			\end{scope}
			\draw[thick,-] (13,3)--(10,7);
			\draw[thick,-] (10,7)--(5,7);
			\begin{scope}[very thick,dashed]
			\draw[thick,-] (5,7)--(0,4);
			\end{scope}
			\draw[thick,-] (0,4)--(8,0);
			\draw[thick,-] (0,4)--(13,3);
			\draw[thick,-] (0,4)--(10,7);
			\draw[thick,-] (3,0)--(13,3);
			\draw[thick,-] (3,0)--(10,7);
			\draw[thick,-] (3,0)--(5,7);
			\draw[thick,-] (8,0)--(10,7);
			\draw[thick,-] (8,0)--(5,7);
			\draw[thick,-] (13,3)--(5,7);
			\filldraw[black] (0,4) circle (2pt);
			\filldraw[black] (3,0) circle (2pt);
			\filldraw[black] (8,0) circle (2pt);
			\filldraw[black] (13,3) circle (2pt);
			\filldraw[black] (10,7) circle (2pt);
			\filldraw[black] (5,7) circle (2pt);
			\put(10,-10){\resizebox{.12\hsize}{!}{$(0,0,2)$}};
			\put(100,-10){\resizebox{.06\hsize}{!}{$a_{1}$}};
			\put(170,30){\resizebox{.09\hsize}{!}{$a_{i-1}$}};
			\put(150,90){\resizebox{.06\hsize}{!}{$a_{i}$}};
			\put(40,95){\resizebox{.09\hsize}{!}{$a_{i+1}$}};
			\put(-20,50){\resizebox{.06\hsize}{!}{$a_{n}$}};
			\put(50,-20){\resizebox{.06\hsize}{!}{$(a)$}};	
			\end{tikzpicture}
		\end{minipage}%
		\begin{minipage}{0.5\textwidth}
			
			\begin{tikzpicture}[thick,scale=0.47, every node/.style={scale=0.6}]
			\draw[thick,-] (0,4)--(3,0);
			\draw[thick,-] (8,0)--(13,3);
			\begin{scope}[very thick,dashed]
			\draw[thick,-] (13,3)--(10,7);
			\end{scope}
			\begin{scope}[very thick,dashed]
			\draw[thick,-] (10,7)--(5,7);
			\end{scope}
			\draw[thick,-] (5,7)--(0,4);
			\filldraw[black] (0,4) circle (2pt);
			\filldraw[black] (3,0) circle (2pt);
			\filldraw[black] (8,0) circle (2pt);
			\filldraw[black] (13,3) circle (2pt);
			\filldraw[black] (10,7) circle (2pt);
			\filldraw[black] (5,7) circle (2pt);
			\put(10,-10){\resizebox{.12\hsize}{!}{$(0,0,3)$}};
			\put(100,-10){\resizebox{.06\hsize}{!}{$a_{1}$}};
			\put(170,30){\resizebox{.06\hsize}{!}{$a_{2}$}};
			\put(150,90){\resizebox{.06\hsize}{!}{$a_{i}$}};
			\put(40,95){\resizebox{.09\hsize}{!}{$a_{n-1}$}};
			\put(-20,50){\resizebox{.06\hsize}{!}{$a_{n}$}};
			\put(50,-20){\resizebox{.06\hsize}{!}{$(b)$}};	
			\put(25,20){\resizebox{.2\hsize}{!}{$\frac{m-k_{n}\cdot(2\alpha-1)}{2}$}};
			\put(30,60){\resizebox{.11\hsize}{!}{$2\alpha+1$}};
			\put(150,15){\resizebox{.11\hsize}{!}{$2\alpha+1$}};
			\end{tikzpicture}
			
		\end{minipage}
		\vspace{0.5 cm}
		\caption{The catenary graph of $\beta_{3}=2n$} \label{fig:9 }
	\end{figure}

\end{proof}
\begin{corollary}
	Let $\mathbb{S}$ be a member of the telescopic numerical semigroup family $\Omega$ given in Theorem \ref{thm2}. The catenary degree of $\mathbb{S}$ is following that:
	$$	c(\mathbb{S})=\max\left( 2\alpha+1,\frac{n-\max\{k\}\cdot(2\alpha-1)}{2}\right) $$
	for $k\leq\frac{n}{2\alpha+1}$, $2|n-k\cdot(2\alpha+1)$ and $k\in\mathbb{N}$.
\end{corollary}	
\begin{exam}
	Let $\mathbb{S}=\langle 6,27,83\rangle \in\Omega$. Then $\beta_{1}=\beta_{2}=54$ and  $\beta_{3}=249$. The factorizations of $\beta_{1}=\beta_{2}=54$ are  $(9,0,0)$ and $(0,0,2)$; the factorizations of $\beta_{3}=249$ are  $(37,1,0)$, $(28,3,0)$, $(19,5,0)$, $(10,7,0)$, $(1,9,0)$ and $(0,0,3)$. However,  $c(\beta_{1})=c(\beta_{2})=c(54)=9$ and $c(\beta_{3})=c(249)=10$. Thus, $c(\mathbb{S})=10$.
\end{exam}
\begin{theorem}\label{thm11}
	Let $\mathbb{S}$ be a member of the telescopic numerical semigroup family $\Psi$ given in Theorem \ref{thm2} such that $\beta_{1},\beta_{2}$  and $\beta_{3}$ be Betti elements of the numerical semigroup $\mathbb{S}$. In this case,
	\begin{itemize}
		\item[i)]If $6\alpha+4<p\leq9\alpha+6$ and $3|p$, then the factorizations of $\beta_{1} $ are $(0,0,2)$ and $(\frac{p-k\cdot(3\alpha+2)}{3},k,0)$  for $ k\leq \frac{p}{3\alpha+2}$, $3|p-k\cdot(3\alpha+2)$ and $k\in\mathbb{N}$. If other, then the factorizations of $\beta_{1}$ are $(3\alpha+2,0,0)$ and $(0,3,0)$.
		\item[ii)] If $6\alpha+4<p\leq9\alpha+6$ and $3|p$, then the factorizations of $\beta_{2}=18\alpha+12$ are $(\frac{(9\alpha+6-p)}{3},0,2)$, $(3\alpha+2,0,0)$ and $(0,3,0)$. If other, then the factorizations of $\beta_{2}$ are $(3\alpha+2,0,0)$ and $(0,3,0)$.
		\item[iii)] The factorizations of $\beta_{3}=2p$ are $(\frac{p-k\cdot(3\alpha+2)}{3},k,0)$ and $(0,0,2)$ for $ k\leq \frac{p}{3\alpha+2}$, $3|p-k\cdot(3\alpha+2)$ and $k\in\mathbb{N}$.
	\end{itemize}
	
\end{theorem}
\begin{proof}
	The Betti elements of the numerical semigroup $\mathbb{S}$, which is a member of the telescopic numerical semigroup family  $\Psi$ given in Theorem \ref{thm2}, are $\beta_{1},\beta_{2}$  and $\beta_{3}$ in the proof of Theorem \ref{thm4}. 
	\begin{itemize}
		\item[i)] Firstly, we will find the factorizations of $\beta_{1}$.
		\begin{itemize}
			\item[a)] If $6\alpha+4<p\leq9\alpha+6$ and $3|p$, then the factorizations of $\beta_{1}=2p$. We write $\beta_{1}=2p=6x_{1}+(6\alpha+4)x_{2}+px_{3} \quad(x_{1},x_{2},x_{3}\in\mathbb{N})$. In this case, it clear that $x_{3}$ is one of $0,1$ or $2$. If $x_{3}=0$, then $x_{1}=\frac{p-(3\alpha+2)\cdot x_{2}}{3}$, $x_{2}=\frac{p-3x_{1}}{3\alpha+2}$, since $x_{1}$ and $x_{2}$ are  nonnegative integers, $x_{2}=k\in\mathbb{N}$ such that $3| p-k\cdot(3\alpha+2)$ and $k\leq\frac{p}{3\alpha+2}$. Thus, If $x_{3}=0$, then the factorization of $\beta_{3}=2p$ is $(\frac{p-k\cdot(3\alpha+2)}{3},k,0)$. If $x_{3}=1$, then the equation $p=6x_{1}+(6\alpha+4)x_{2}$ is obtained. But this contradicts that $p$ is an odd integer. If $x_{3}=2$, then we write $0=6x_{1}+(6\alpha+4)x_{2}$. Hence, it is clear that $x_{1}=0$ and $x_{2}=0$. Thus, the factorization of $\beta_{1}=2p$ is $(0,0,2)$.
			\item[b)]If $6\alpha+4<p\leq9\alpha+6$ and $3|p$, then the factorizations of $\beta_{1}=18\alpha+12$. We write Thus, $\beta_{1}=18\alpha+12=6x_{1}+(6\alpha+4)x_{2}+px_{3} \quad(x_{1},x_{2},x_{3}\in\mathbb{N})$. Since $18\alpha+12$ is a nonnegative even  integer, $x_{3}$ must be a positive even  integer, too. Furthermore, since $p>6\alpha+4$ it should be $x_{3}=0$ or $x_{3}=2$. If $x_{3}=0$, then $x_{2}=3-\frac{3x_{1}}{3\alpha+2}$. Since $x_{1},x_{2}\in\mathbb{N}$, $x_{1}=0$ or $x_{1}=3\alpha+2$. Therefore, if $x_{3}=0$, then the factorizations of $\beta_{1}=18\alpha+12$ are $(3\alpha+2,0,0)$ and $(0,3,0)$. If $x_{3}=2$, then $x_{2}=3-\frac{3x_{1}+p}{3\alpha+2}$. Since $x_{1},x_{2}\in\mathbb{N}$, the fraction $\frac{3x_{1}+p}{3\alpha+2}$ is one of $0,1,2,$ or $3$. If $\frac{3x_{1}+p}{3\alpha+2}=0$, then $p=-3x_{1}$. But this statement contradicts the acceptance of $x_{1}$  and $p$.  If $\frac{3x_{1}+p}{3\alpha+2}=1$, then $x_{1}=3-\frac{(3\alpha+2)-p}{3}$. But since $p>6\alpha+4$, $x_{1}\notin\mathbb{N}$ is a contradiction. If $\frac{3x_{1}+p}{3\alpha+2}=2$ and $\frac{3x_{1}+p}{3\alpha+2}=3$, then  the similar contradiction is obtained.
		\end{itemize}	
		\item[ii)] We will find the factorizations of $\beta_{2}=18\alpha+12$. We write $\beta_{2}=18\alpha+12=6x_{1}+(6\alpha+4)x_{2}+px_{3} \quad(x_{1},x_{2},x_{3}\in\mathbb{N})$. Since $18\alpha+12$ is a nonnegative even  integer, $x_{3}$ must be a positive even  integer, too. Furthermore, since $p>6\alpha+4$ it should be $x_{3}=0$ or $x_{3}=2$. If $x_{3}=0$, then $x_{2}=3-\frac{3x_{1}}{3\alpha+2}$. Since $x_{1},x_{2}\in\mathbb{N}$, $x_{1}=0$ or $x_{1}=3\alpha+2$. Therefore, if $x_{3}=0$, then the factorizations of $\beta_{1}=18\alpha+12$ are $(3\alpha+2,0,0)$ and $(0,3,0)$. If $x_{3}=2$, then $x_{2}=3-\frac{3x_{1}+p}{3\alpha+2}$. Since $x_{1},x_{2}\in\mathbb{N}$, the fraction $\frac{3x_{1}+p}{3\alpha+2}$ is one of $0,1,2,$ or $3$. When the fraction  $\frac{3x_{1}+p}{3\alpha+2}$ is one of $0,1,$ or $2$. If $\frac{3x_{1}+p}{3\alpha+2}=3$, then $x_{1}=\frac{(9\alpha+6)-p}{3}$. So, $x_{1}\in\mathbb{N}$ if and only if $6\alpha+4<p\leq9\alpha+6$ and $3|p$. Thus, if $6\alpha+4<p\leq9\alpha+6$  and $3|p$, then the factorizations of $\beta_{2}=18\alpha+12$ is $(\frac{(9\alpha+6-p)}{3},0,2)$.
		\item[iii)] Now let’s find the factorizations of $\beta_{3}=2p$. We write $\beta_{3}=2p=6x_{1}+(6\alpha+4)x_{2}+px_{3} \quad(x_{1},x_{2},x_{3}\in\mathbb{N})$. In this  case, it is clear that $x_{3}$ is one of $0,1$ or $2$. If $x_{3}=0$, then $x_{1}=\frac{p-(3\alpha+2)x_{2}}{3}$ and $x_{2}=\frac{p-3x_{1}}{3\alpha+2}$. Since $x_{1}$ and $x_{2}$ are  nonnegative integers, $x_{2}=k\in\mathbb{N}$ such that $3| p-k\cdot(3\alpha+2)$ and $k\leq\frac{p}{3\alpha+2}$. Thus, If $x_{3}=0$, then the factorization of $\beta_{3}=2p$ is $(\frac{p-k\cdot(3\alpha+2)}{3},k,0)$. If $x_{3}=1$, then the equation $p=6x_{1}+(6\alpha+4)x_{2}$ is obtained. But this contradicts that $p$ is an odd integer. If $x_{3}=2$, then we write $0=6x_{1}+(6\alpha+4)x_{2}$. Hence, it is clear that $x_{1}=0$ and $x_{2}=0$. Thus, the factorization of $\beta_{3}=2p$ is $(0,0,2)$.	
		
	\end{itemize}	
	
\end{proof}

\begin{theorem}\label{thm12} Let $\mathbb{S}$ be a member of the telescopic numerical semigroup family  $\Psi$ given in Theorem \ref{thm2}. The catenary degree of $\mathbb{S}$ is following that: 
	\begin{itemize}	
		\item[i)] \begin{displaymath}
		c(\beta_{1})=\left\{ \begin{array}{ll}
		\dfrac{p}{3} & \textrm{if } 6\alpha+4<p\leq 9\alpha+6 \textrm{ and } 3|p \\
		3\alpha+2 & \textrm{if other}
		\end{array} \right.
		\end{displaymath}
		\item[ii)] \begin{displaymath}
		c(\beta_{2})=\left\{ \begin{array}{ll}
		\dfrac{p}{3} & \textrm{if } 6\alpha+4<p\leq 9\alpha+6 \textrm{ and } 3|p \\
		3\alpha+2 & \textrm{if other}
		\end{array} \right.
		\end{displaymath}
		\item[iii)] $c(\beta_{3})=\max\{3\alpha+2,\frac{p-\max\{k\}\cdot (3\alpha-1)}{3}\} $ for $ k\leq \frac{p}{3\alpha+2}$, $3|p-k\cdot(3\alpha+2)$ and $k\in\mathbb{N}$.
	\end{itemize}	
\end{theorem}
\begin{proof} Assume that $\mathbb{S}$ is a member of the telescopic numerical semigroup family $\Psi$  given in Theorem \ref{thm2}. From the proof of Theorem \ref{thm4}, we know the Betti element of the numerical semigroup $\mathbb{S}$. Moreover, the factorizations of the betti elements of  $\mathbb{S}$ are given in Theorem \ref{thm11}.
	
	\begin{itemize}
		\item[i)] We will find  the catenary degree of $\beta_{1}$.
		\begin{itemize}
			\item[a)] From Theorem \ref{thm11}, $6\alpha+4<p\leq9\alpha+6$  and $3|p$, then the factorizations of $\beta_{1}=2p$ are $(\frac{p-k\cdot(3\alpha+2)}{3},k,0)$ and $(0,0,2)$ for $ k\leq \frac{p}{3\alpha+2}$, $3|p-k\cdot(3\alpha+2)$ and $k\in\mathbb{N}$. If $6\alpha+4<p<9\alpha+6$  and $3|p$, then $k=0$. Thus, then the factorizations of $\beta_{1}=2p$ are $(0,0,2)$ and $(\frac{p}{3},0,0)$. If $p=9\alpha+6$, then $k=0$ or $k=3$. Accordingly, the factorizations of $\beta_{1}=2p$  are $(0,0,2)$, $(0,3,0)$ and $(\frac{p}{3},0,0)$. In this case, the lengths of the edges  between these factorizations are as follows:
			$$\gcd\{(0,0,2),(\dfrac{p}{3},0,0)\}=(0,0,0)$$
			$$\gcd\{(0,0,2),(0,3,0)\}=(0,0,0)$$
			$$\gcd\{(0,3,0),(\dfrac{p}{3},0,0)\}=(0,0,0)$$
			and
			$$dist\{(0,0,2),(\dfrac{p}{3},0,0)\}=\dfrac{p}{3}$$
			$$dist\{(0,0,2),(0,3,0)\}=3$$
			$$dist\{(0,3,0),(\dfrac{p}{3},0,0)\}=\dfrac{p}{3}$$	
			When each vertex is labeled with one of the factorizations of $\beta_{1}=2p$ and each edge is labeled with distance between the factorizations of $\beta_{1}=2p$ at either end, we get Figure \ref{fig:10} and  Figure \ref{fig:11}. Hence, if we draw the graphs in Figure \ref{fig:10} and  Figure \ref{fig:11} which consist of these vertices and edges, then the catenary degree of $\beta_{1}=2p$ is $\dfrac{p}{3}$.\\ 	\indent	
			When $6\alpha+4<p<9\alpha+6$  and $3|p$, we get Figure \ref{fig:10}
			\\
			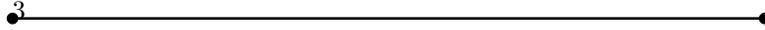
\begin{figure}[H]
				\centering
				\begin{tikzpicture}
				\draw[thick,-] (0,0)--(10,0);
				\draw[thick,-] (10,0)--(0,0);
				\filldraw[black] (0,0) circle (2pt);
				\filldraw[black] (10,0) circle (2pt);
				\put(-20,5){$(0,0,2)$};
				\put(270,10){$(\dfrac{p}{3},0,0)$};
				\put(150,10){$\dfrac{p}{3}$};
				\end{tikzpicture}
				\vspace{0.5 cm}
				\caption{The catenary graph of $\beta_{1}=2p$ with factorizations $(0,0,2)$ and $(\dfrac{p}{3},0,0)$ } \label{fig:10}
			\end{figure}
			When $p=9\alpha+6$, we get Figure \ref{fig:11}
			\\
			\begin{figure}[H]
				\centering
				\begin{minipage}{.5\textwidth}
					\centering
					\centering
					\begin{tikzpicture}[thick,scale=0.47, every node/.style={scale=0.6}]
					\draw[thick,-] (0,0)--(10,0);
					\draw[thick,-] (10,0)--(5,5);
					\draw[thick,-] (5,5)--(0,0);
					\filldraw[black] (0,0) circle (2pt);
					\filldraw[black] (10,0) circle (2pt);
					\filldraw[black] (5,5) circle (2pt);
					\put(-20,-8){\resizebox{.09\hsize}{!}{$(0,3,0)$}};
					\put(120,-8){\resizebox{.09\hsize}{!}{$(0,0,2)$}};
					\put(50,5){\resizebox{.02\hsize}{!}{$3$ }};
					\put(50,75){\resizebox{.15\hsize}{!}{$(\dfrac{p}{3},0,0)$}};
					\put(20,35){\resizebox{.04\hsize}{!}{$\frac{p}{3}$}};
					\put(100,35){\resizebox{.04\hsize}{!}{$\frac{p}{3}$}};
					\put(50,-20){\resizebox{.04\hsize}{!}{$(a)$}};
					\end{tikzpicture}
				\end{minipage}%
				\begin{minipage}{0.5\textwidth}
					
					\begin{tikzpicture}[thick,scale=0.47, every node/.style={scale=0.6}]
					\draw[thick,-] (0,0)--(0,0);
					\draw[thick,-] (0,0)--(5,5);
					\draw[thick,-] (5,5)--(10,0);
					\filldraw[black] (0,0) circle (2pt);
					\filldraw[black] (10,0) circle (2pt);
					\filldraw[black] (5,5) circle (2pt);
					\put(-20,-12){\resizebox{.09\hsize}{!}{$(0,3,0)$}};
					\put(120,-8){\resizebox{.09\hsize}{!}{$(0,0,2)$}};
					\put(50,75){\resizebox{.15\hsize}{!}{$(\dfrac{p}{3},0,0)$}};
					\put(20,35){\resizebox{.04\hsize}{!}{$\frac{p}{3}$}};
					\put(100,35){\resizebox{.04\hsize}{!}{$\frac{p}{3}$}};
					\put(50,-20){\resizebox{.04\hsize}{!}{$(b)$}};
					\end{tikzpicture}
				\end{minipage}
				\vspace{0.5 cm}
				\caption{The catenary graph of $\beta_{1}=2p$ with factorizations $(0,0,2)$, $(0,3,0)$ and $(\dfrac{p}{3},0,0)$ } \label{fig:11}
			\end{figure}
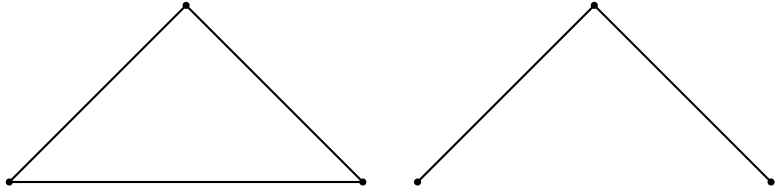
			
			\item[b)] From Theorem \ref{thm11}, if other, then the factorizations of $\beta_{1}=18\alpha+12$ are $(3\alpha+2,0,0)$ and $(0,3,0)$. In this case, the lengths of the edges  between these factorizations are as follows:
			$$\gcd\{(3\alpha+2,0,0),(0,3,0)\} = (0,0,0)	$$
			and
			$$dist\{(3\alpha+2,0,0),(0,3,0)\}=3\alpha+2.$$
			When we draw the graph in Figure \ref {fig:12} which is constituted by edges these connect vertices points, the  catenary degree of $\beta_{1}$ is  $3\alpha+2$.
			\\
			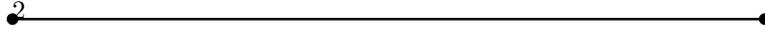
\begin{figure}[H]
				\centering
				\begin{tikzpicture}
				\draw[thick,-] (0,0)--(10,0);
				\draw[thick,-] (10,0)--(0,0);
				\filldraw[black] (0,0) circle (2pt);
				\filldraw[black] (10,0) circle (2pt);
				\put(-20,5){$\left( 3\alpha+2, 0, 0 \right)$};
				\put(270,5){$\left( 0, 3, 0 \right)$};
				\put(150,5){$3\alpha +2$};
				\end{tikzpicture}
				\vspace{0.5 cm}
				\caption{The catenary graph of $\beta_{1}=18\alpha+12$ with factorizations $(3\alpha+1,0,0)$ and $(0,3,0)$} \label{fig:12}
			\end{figure}
			
		\end{itemize}	
		\item[ii)] We will find  the catenary degree of $\beta_{2}=18\alpha+12$. 
		\begin{itemize}		
			\item[a)] If $6\alpha+4<p\leq9\alpha+6$  and $3|p$, then the factorizations of $\beta_{2}=18\alpha+12$ are $(\frac{(9\alpha+6-p)}{3},0,2)$, $(3\alpha+2,0,0)$ and $(0,3,0)$. In this case, the lengths of the edges  between these factorizations are as follows:
			$$\gcd\{(\frac{(9\alpha+6-p)}{3},0,2),(3\alpha+2,0,0)\}=(\frac{(9\alpha+6-p)}{3},0,0) $$	
			$$dist\{(\frac{(9\alpha+6-p)}{3},0,2),(3\alpha+2,0,0)\}=\frac{p}{3}$$	
			and
			
			$$\gcd\{(\frac{(9\alpha+6-p)}{3},0,2),(0,3,0)\}=(0,0,0) $$
			$$dist\{(\frac{(9\alpha+6-p)}{3},0,2),(0,3,0)\}=4+3\alpha-\frac{p}{3}$$
			and
			
			$$\gcd\{(3\alpha+2,0,0),(0,3,0)\} = (0,0,0)	$$
			
			$$dist\{(3\alpha+2,0,0),(0,3,0)\}=3\alpha+2.$$
			
			When each vertex is labeled with one of the factorizations of $\beta_{2}=18\alpha+12$  and each edge is labeled with distance between the factorizations of $\beta_{2}=18\alpha+12$  at either end, we get Figure \ref{fig:13} $(a)$.	Hence, if we draw the graph in Figure \ref{fig:13} $(a)$, which consists of  these vertices and edges, the catenary degree of $\beta_{2}=18\alpha+12$ is $\frac{p}{3}$ by Figure \ref{fig:13} $(b)$. Because $3\alpha+2\geq\frac{p}{3}>4+3\alpha-\frac{p}{3}$ for all $\alpha \in \mathbb{N}$ and $p\in \mathbb{N}_{o}$. 
			\\
			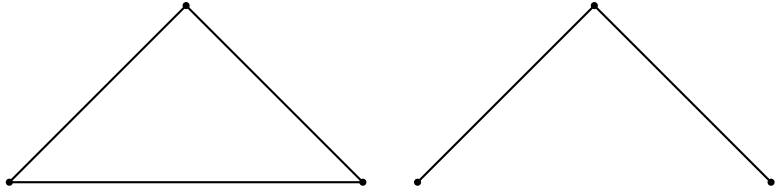
\begin{figure}[H]
				\centering
				\begin{minipage}{.5\textwidth}
					\centering
					\centering
					\begin{tikzpicture}[thick,scale=0.47, every node/.style={scale=0.6}]
					\draw[thick,-] (0,0)--(10,0);
					\draw[thick,-] (10,0)--(5,5);
					\draw[thick,-] (5,5)--(0,0);
					\filldraw[black] (0,0) circle (2pt);
					\filldraw[black] (10,0) circle (2pt);
					\filldraw[black] (5,5) circle (2pt);
					\put(-20,-8){\resizebox{.15\hsize}{!}{$\left( 3\alpha+2, 0, 0 \right)$}};
					\put(120,-8){\resizebox{.09\hsize}{!}{$\left( 0, 3, 0 \right)$}};
					\put(50,5){\resizebox{.09\hsize}{!}{$ 3\alpha +2 $ }};
					\put(35,75){\resizebox{.25\hsize}{!}{$\left( \frac{(9\alpha+6-p)}{3}, 0, 2 \right)$}};
					\put(20,35){\resizebox{.03\hsize}{!}{$\frac{p}{3}$}};
					\put(100,35){\resizebox{.15\hsize}{!}{$4+3\alpha-\frac{p}{3}$}};
					\put(50,-20){\resizebox{.04\hsize}{!}{$(a)$}};
					\end{tikzpicture}
				\end{minipage}%
				\begin{minipage}{0.5\textwidth}
					
					\begin{tikzpicture}[thick,scale=0.47, every node/.style={scale=0.6}]
					\draw[thick,-] (0,0)--(0,0);
					\draw[thick,-] (0,0)--(5,5);
					\draw[thick,-] (5,5)--(10,0);
					\filldraw[black] (0,0) circle (2pt);
					\filldraw[black] (10,0) circle (2pt);
					\filldraw[black] (5,5) circle (2pt);
					\put(-20,-12){\resizebox{.15\hsize}{!}{$\left( 3\alpha+2, 0, 0 \right)$}};
					\put(120,-8){\resizebox{.09\hsize}{!}{$\left( 0, 3, 0 \right)$}};
					\put(35,75){\resizebox{.25\hsize}{!}{$\left( \frac{(9\alpha+6-p)}{3}, 0, 2 \right)$}};
					\put(20,35){\resizebox{.03\hsize}{!}{$\frac{p}{3}$}};
					\put(100,35){\resizebox{.15\hsize}{!}{$4+3\alpha-\frac{p}{3}$}};
					\put(50,-20){\resizebox{.04\hsize}{!}{$(b)$}};
					\end{tikzpicture}
				\end{minipage}
				\vspace{0.5 cm}
				\caption{The catenary graph of $\beta_{2}=18\alpha+12$ with factorizations $(\frac{(9\alpha+6-p)}{3},0,2)$, $(3\alpha+2,0,0)$ and $(0,3,0)$} \label{fig:13}
			\end{figure}
			
			\item[b)] If other, then the factorizations of $\beta_{2}=18\alpha+12$ are $(3\alpha+2,0,0)$ and $(0,3,0)$ by Theorem \ref{thm11}. In this case, the length of the edge between these factorizations is found as
			
			$$\gcd\{(3\alpha+2,0,0),(0,3,0)\} = (0,0,0)$$
			
			$$dist\{(3\alpha+2,0,0),(0,3,0)\}=3\alpha+2.$$
			
			When we draw the graph in Figure \ref {fig:14} which is constituted by edges these connect vertices points, the  catenary degree of $\beta_{2}$ is $3\alpha+2$.
			\\
			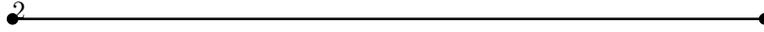
\begin{figure}[H]
				\centering
				\begin{tikzpicture}
				\draw[thick,-] (0,0)--(10,0);
				\draw[thick,-] (10,0)--(0,0);
				\filldraw[black] (0,0) circle (2pt);
				\filldraw[black] (10,0) circle (2pt);
				\put(-20,5){$\left( 3\alpha+2, 0, 0 \right)$};
				\put(270,5){$\left( 0, 3, 0 \right)$};
				\put(150,5){$3\alpha +2$};
				\end{tikzpicture}
				\vspace{0.5 cm}
				\caption{The catenary graph of $\beta_{1}=18\alpha+12$ with factorizations $(3\alpha+1,0,0)$ and $(0,3,0)$} \label{fig:14}
			\end{figure}
		\end{itemize}
		\item[iii)] Finally, we will find catenary degree $\beta_{3}=2p$. The factorizations of $\beta_{3}=2p$ are $(\frac{p-k\cdot(3\alpha+2)}{3},k,0)$ and $(0,0,2)$ for $ k\leq \frac{p}{3\alpha+2}$, $3|p-k\cdot(3\alpha+2)$ and $k\in\mathbb{N}$ by Theorem \ref{thm11}. Thus, in this case, the length of the edge between these factorizations is found as:\\ 	\indent 	Since there will be an edge for every nonnegative integer $k$, let's show that the edge corresponding to each $k_{i}$ with $a_{i}$ such that $a_{i}=(\frac{p-k_{i}(3\alpha-2)}{3},k_{i},0)$ for $i\in \left\lbrace 1,2,\dots,n\right\rbrace $. Where $k_{1}<k_{2}<\dots<k_{n}$
		$$\gcd\{a_{i},(0,0,2)\}=(0,0,0)$$
		and
		$$dist\{a_{i},(0,0,2)\}=\frac{p-k_{i}(3\alpha+2)}{3}+k_{i}=\frac{p-k_{i}(3\alpha-1)}{3}$$
		Let $i\in \left\lbrace 1,2,\dots,n-1\right\rbrace $ and $j\in \left\lbrace 2,3,\dots,n\right\rbrace $ such that $i<j $.
		
		$$\gcd\{a_{i},a_{j}\}=(\frac{p-k_{j}(3\alpha+2)}{3},k_{i},0)$$
		and
		$$dist\{a_{i},a_{j}\}=\max\{|\frac{(k_{j}-k_{i})(3\alpha+2)}{3}|,|k_{j}-k_{i}|\}=\frac{(k_{j}-k_{i})(3\alpha+2)}{3}$$
		The following equations are resulted from those obtained above:  \\ 	\indent for $i\in \left\lbrace 1,2,\dots,n\right\rbrace $.
		\begin{multline*}
		\min\{dist\{a_{1},(0,0,2)\},dist\{a_{2},(0,0,2)\},\dots,dist\{a_{n},(0,0,2)\}\} = \min\{\frac{p-k_{1}(3\alpha-1)}{3}, \\
		\frac{p-k_{2}(3\alpha-1)}{3},\dots,\frac{p-k_{n}(3\alpha-1)}{3}\}=\frac{p-k_{n}(3\alpha-1)}{3}=dist\{a_{n},(0,0,2)\}
		\end{multline*}
		and
		\begin{multline*}
		\min\{dist\{a_{i},a_{1}\},dist\{a_{i},a_{2}\},\dots,dist\{a_{i},a_{i-1}\},dist\{a_{i},a_{i+1}\},\dots,dist\{a_{i},a_{n}\}\}= \\ \min\{|\frac{(k_{i}-k_{1})(3\alpha+2)}{3}|,|\frac{(k_{i}-k_{2})(3\alpha+2)}{3}|,\dots,|\frac{(k_{i}-k_{i-1})(3\alpha+2)}{3}|,|\frac{(k_{i}-k_{i+1})(3\alpha+2)}{3}|,\\
		\dots,|\frac{(k_{n}-k_{i})(3\alpha+2)}{3}|\} =|\frac{(k_{i}-k_{i-1})(3\alpha+2)}{3}|=|\frac{(k_{i+1}-k_{i})(3\alpha+2)}{3}|\\
		=3\alpha+2=dist\{a_{i},a_{i-1}\}=dist\{a_{i},a_{i+1}\}
		\end{multline*}
		When each vertex is labeled with one of the factorizations of $\beta_{3}=2p$  and each edge is labeled with distance between the factorizations of $\beta_{3}=2p$   at either end, we get Figure \ref{fig:15 }$(a)$. If vertices with maximal length are removed from the connected graph in Figure \ref{fig:15 }$(a)$, then Figure\ref{fig:15 }$(b)$ is obtained. Thus, the catenary degree of $\beta_{3}=2p$ is $\max\left( 3\alpha+2,\frac{p-\max\{k\}(3\alpha-1)}{3}\right) $ for $ k\leq \frac{p}{3\alpha+2}$, $3|p-k\cdot(3\alpha+2)$ and $k\in\mathbb{N}$.
		\\
		\begin{figure}[H]
			\centering
			\begin{minipage}{.5\textwidth}
				\centering
				\centering
				\begin{tikzpicture}[thick,scale=0.47, every node/.style={scale=0.6}]
				\draw[thick,-] (0,4)--(3,0);
				\draw[thick,-] (3,0)--(8,0);
				\begin{scope}[very thick,dashed]
				\draw[thick,-] (8,0)--(13,3);
				\end{scope}
				\draw[thick,-] (13,3)--(10,7);
				\draw[thick,-] (10,7)--(5,7);
				\begin{scope}[very thick,dashed]
				\draw[thick,-] (5,7)--(0,4);
				\end{scope}
				\draw[thick,-] (0,4)--(8,0);
				\draw[thick,-] (0,4)--(13,3);
				\draw[thick,-] (0,4)--(10,7);
				\draw[thick,-] (3,0)--(13,3);
				\draw[thick,-] (3,0)--(10,7);
				\draw[thick,-] (3,0)--(5,7);
				\draw[thick,-] (8,0)--(10,7);
				\draw[thick,-] (8,0)--(5,7);
				\draw[thick,-] (13,3)--(5,7);
				\filldraw[black] (0,4) circle (2pt);
				\filldraw[black] (3,0) circle (2pt);
				\filldraw[black] (8,0) circle (2pt);
				\filldraw[black] (13,3) circle (2pt);
				\filldraw[black] (10,7) circle (2pt);
				\filldraw[black] (5,7) circle (2pt);
				\put(10,-10){\resizebox{.12\hsize}{!}{$(0,0,2)$}};
				\put(100,-10){\resizebox{.06\hsize}{!}{$a_{1}$}};
				\put(170,30){\resizebox{.09\hsize}{!}{$a_{i-1}$}};
				\put(150,90){\resizebox{.06\hsize}{!}{$a_{i}$}};
				\put(40,95){\resizebox{.09\hsize}{!}{$a_{i+1}$}};
				\put(-20,50){\resizebox{.06\hsize}{!}{$a_{n}$}};
				\put(50,-20){\resizebox{.06\hsize}{!}{$(a)$}};	
				\end{tikzpicture}
			\end{minipage}%
			\begin{minipage}{0.5\textwidth}
				
				\begin{tikzpicture}[thick,scale=0.47, every node/.style={scale=0.6}]
				\draw[thick,-] (0,4)--(3,0);
				\draw[thick,-] (8,0)--(13,3);
				\begin{scope}[very thick,dashed]
				\draw[thick,-] (13,3)--(10,7);
				\end{scope}
				\begin{scope}[very thick,dashed]
				\draw[thick,-] (10,7)--(5,7);
				\end{scope}
				\draw[thick,-] (5,7)--(0,4);
				\filldraw[black] (0,4) circle (2pt);
				\filldraw[black] (3,0) circle (2pt);
				\filldraw[black] (8,0) circle (2pt);
				\filldraw[black] (13,3) circle (2pt);
				\filldraw[black] (10,7) circle (2pt);
				\filldraw[black] (5,7) circle (2pt);
				\put(10,-10){\resizebox{.12\hsize}{!}{$(0,0,2)$}};
				\put(100,-10){\resizebox{.06\hsize}{!}{$a_{1}$}};
				\put(170,30){\resizebox{.06\hsize}{!}{$a_{2}$}};
				\put(150,90){\resizebox{.06\hsize}{!}{$a_{i}$}};
				\put(40,95){\resizebox{.09\hsize}{!}{$a_{n-1}$}};
				\put(-20,50){\resizebox{.06\hsize}{!}{$a_{n}$}};
				\put(50,-20){\resizebox{.06\hsize}{!}{$(b)$}};	
				\put(25,20){\resizebox{.2\hsize}{!}{$\frac{p-\max\{k\}(3\alpha-1)}{3}$}};
				\put(30,60){\resizebox{.11\hsize}{!}{$3\alpha+2$}};
				\put(150,15){\resizebox{.11\hsize}{!}{$3\alpha+2$}};
				\end{tikzpicture}
				
			\end{minipage}
			\vspace{0.5 cm}
			\caption{The catenary graph of $\beta_{3}=2p$} \label{fig:15 }
		\end{figure}
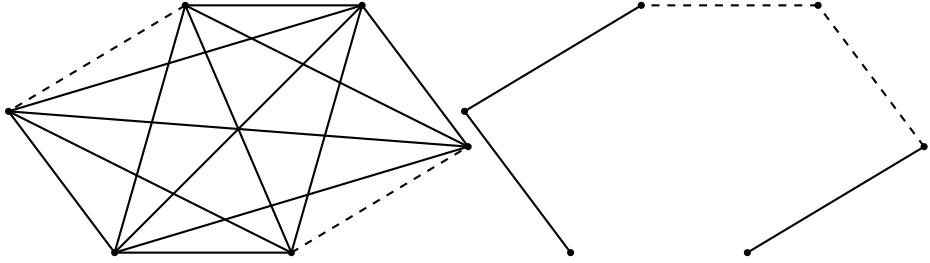
		
	\end{itemize}
	
\end{proof}
\begin{corollary}
	Let $\mathbb{S}$ be a member of the telescopic numerical semigroup family $\Psi$  given in Theorem \ref{thm2}. The catenary degree of $\mathbb{S}$ is following that:
	$$	c(\mathbb{S})=\max\left( 3\alpha+2,\frac{p-\max\{k\}\cdot(3\alpha-1)}{3}\right) $$
	for $ k\leq \frac{p}{3\alpha+2}$, $3|p-k\cdot(3\alpha+2)$ and $k\in\mathbb{N}$.
\end{corollary}	
\begin{exam}
	Let $\mathbb{S}=\langle 6,34,39\rangle \in\Psi$. Then $\beta_{1}=\beta_{3}=78$ and  $\beta_{2}=102$. The factorizations of $\beta_{1}=\beta_{3}=78$ are  $(13,0,0)$ and $(0,0,2)$; the factorizations of $\beta_{2}=102$ are  $(17,0,0)$, $(4,0,2)$ and $(0,0,3)$. However,  $c(\beta_{1})=c(\beta_{2})=c(\beta_{3})=13$. Thus, $c(\mathbb{S})=13$.
\end{exam}

\end{document}